\documentclass[poms,final,nonblindrev]{poms1_V1} 

\OneAndAHalfSpacedXI 

\newcommand{\sk}[1]{{\textcolor{black}{#1}}}
\newcommand{\hll}[1]{{\textcolor{black}{#1}}}

\usepackage{graphicx}
\usepackage{subfigure, epsfig}
\usepackage{natbib}
\usepackage{algorithm}
\usepackage{algpseudocode}
\usepackage{tikz}
\usepackage{enumitem}
\usepackage{amsmath,amssymb,amsfonts}
\usepackage{rotating}
\usepackage{fancyvrb}
\usepackage{array}    
\usepackage{lscape}
\usepackage{multirow} 
\usepackage{xcolor}
\usepackage{soul}

\sethlcolor{yellow} 
 \bibpunct[, ]{(}{)}{,}{a}{}{,}%
 \def\bibfont{\small}%
\TheoremsNumberedThrough     
\ECRepeatTheorems

\EquationsNumberedThrough    

\usepackage{xcolor}
\begin{document}


\RUNAUTHOR{Title}

\RUNTITLE{Stochastic appointment scheduling with patient-and-time-dependent probability distributions}

\TITLE{Stochastic Appointment Scheduling with Patient-and-Time-Dependent Probability Distributions}


\ARTICLEAUTHORS{%
\AUTHOR{Soheyl Khalilpourazari}
\AFF{Department of Mechanical, Industrial \& Aerospace Engineering, Concordia University, Montreal, Canada\\
Interuniversity Research Centre on Enterprise Networks, Logistics and Transportation (CIRRELT), Montreal, Canada \EMAIL{soheyl.khalilpourazari@mail.concordia.ca}} 
\AUTHOR{Hossein Hashemi Doulabi}
\AFF{Department of Mechanical, Industrial \& Aerospace Engineering, Concordia University, Montreal, Canada\\
Interuniversity Research Centre on Enterprise Networks, Logistics and Transportation (CIRRELT), Montreal, Canada \EMAIL{hossein.hashemi@concordia.ca }}
} 

\ABSTRACT{\hll{Outpatient appointment scheduling, in settings ranging from primary care practices to specialty consultation clinics, requires balancing provider idle time, overtime, and patient waiting time. Uncertainty in service times, patient punctuality, and attendance makes scheduling a challenging task, resulting in high clinic costs and low patient satisfaction. Addressing these uncertain factors is particularly difficult because they follow patient-and-time-dependent probability distributions. Existing models fail to simultaneously incorporate such distributions and often assume identical distributions for all patients in all time slots, mainly due to the increased modeling and solution complexity. In this research, we propose a novel stochastic programming model that captures an exponential number of scenarios using a polynomial number of variables and constraints without relying on sampling methods. We incorporate patient-dependent probability distributions for service times, and patient-and-time-dependent distributions for no-shows and arrival times, and we discuss how these distributions can be estimated from routinely collected electronic health record data. To enhance the model, we explore the impact of personalized reminders on no-show rates and scheduling effectiveness. Besides, we investigate how changes in attendance rates, influenced by incentives or negative clinic experiences, impact the appointment scheduling efficiency. The results demonstrate that, using our proposed model, we optimally solve several instances with up to 14 patients in a reasonable computational time. The generated schedules substantially outperform classical Bailey-type appointment rules, whose best-performing variant increases the total expected cost by 97\% on average relative to our schedules. Moreover, the generated schedules reduce total costs by 34\% on average by incorporating patient-dependent service times, 12\% by considering patient-and-time-dependent unpunctuality, and 67\% by integrating patient-and-time-dependent no-shows. Furthermore, we show that personalized reminders have the potential to reduce total costs by 23\%, and sensitivity analyses confirm that these benefits are robust to moderate estimation errors in the input distributions. Our approach has a significant potential to improve the efficiency of healthcare services by reducing patient waiting time, provider idle time, and overtime costs, and it highlights the value of personalized communication strategies in improving patient attendance.}}%


\maketitle

%


\section{Introduction}
\noindent Ensuring inclusive, affordable, and timely access to high-quality healthcare has become a critical concern in today's society. To achieve this, continuous improvement in healthcare is essential, particularly in enhancing the patient experience. This includes not only the direct medical care received but also the overall journey patients navigate through the healthcare system \citep{Dai2020, Samorani2022Overbooked}. In this context, Appointment Scheduling (AS) is an essential tool for reducing patient wait times, provider idle times, and clinic overtime. However, the effectiveness of appointment scheduling is frequently undermined by uncertainties in service times, patient unpunctuality, and patient no-shows.

\hll{This work focuses on intra-day appointment scheduling for outpatient consultative care, in which a single healthcare provider serves a sequence of booked patients during one clinic session. Representative settings include primary care practices, specialty consultation clinics such as cardiology or endocrinology, and hospital outpatient departments offering follow-up visits. As an illustration, consider a family physician whose morning session includes twelve booked patients: a new patient with several chronic conditions may require a long initial consultation, whereas a returning patient may need only a brief follow-up. Some of these patients arrive early or late depending on rush-hour traffic, and others miss the appointment altogether, so the realized session rarely unfolds as planned. More broadly, the proposed framework applies to any single-server appointment-based service in which customers are served sequentially and service durations, punctuality, and attendance are uncertain. This broad class includes diagnostic imaging, infusion and dialysis sessions, dental and allied-health visits, and even non-medical services such as legal or financial consultations.}

\hll{Appointments of this kind are delivered at a vast scale. In the United States alone, patients made an estimated 1.0 billion office-based physician visits in 2019, equivalent to roughly three visits per person per year \citet{ashman2023characteristics}, across more than 209{,}000 physician offices \citep{USCensusBureau2022CBP}. In England, general practices delivered approximately 383 million appointments in the twelve months ending June 2025 \citep{NHSEngland2025Appointments}. At this scale, attendance failures are costly: a systematic review covering 105 studies reports an average outpatient no-show rate of about 23\% \citep{Dantas2018NoShows}, and \citet{Pesata1999Descriptive} found that 14{,}000 missed appointments in a single midwestern children's hospital led to over one million dollars in losses. These figures indicate that even modest per-session improvements in schedule quality can translate into substantial system-wide savings.}

In this work, we study an appointment scheduling problem under uncertainty in service times, patient arrival times, and no-show behavior. \hll{Unlike previous studies, our model considers patient-dependent stochastic service times, and patient-and-time-dependent unpunctuality and no-shows in a more comprehensive manner. Each of these uncertainty sources has typically been simplified in existing models. Service durations are often treated as identically distributed across patients, or are differentiated only through coarse classifications that cannot capture the diversity of patients' health conditions. Patient unpunctuality, when modeled at all, is usually represented by a single random variable that ignores both individual arrival behavior and the appointment time, even though factors such as rush-hour traffic make lateness strongly time-dependent (Jiang et~al.\ 2019, Zacharias and Yunes 2020). Similarly, no-shows are commonly described by one probability shared by all patients and all time slots, despite empirical evidence that attendance varies systematically across patients and appointment times (\citealp{Luo2012Appointment,Daggy2010,Bennett2009Effect}). These simplifications persist because jointly modeling patient-and-time-dependent uncertainty makes existing formulations computationally intractable, as we review in detail in Section~2. Our research overcomes this limitation by integrating patient-dependent service times and patient-and-time-dependent unpunctuality and no-shows into a single exact model.}

In real-world appointment scheduling problems, service times are patient-dependent, and patient unpunctuality and no-shows are patient-and-time-dependent. Many researchers have studied appointment scheduling with some of these complicating criteria and details separately. However, in none of the studies above, we observe all uncertainty factors with explained details simultaneously \citep{Kong2020,Salzarulo2016Beyond,Li2019I,zacharias2020}. This is because considering these factors makes the modeling complex and cannot be solved optimally by available optimization methodologies. As a remedy to these problems, we propose a novel stochastic mixed-integer programming model that simultaneously considers patient-dependent stochastic service times, and patient-and-time-dependent unpunctuality and no-shows. In addition, the proposed model considers arbitrary probability distributions for uncertain parameters without limiting the probability distributions to be of the same type and captures these distributions as they are without any sampling. The model benefits from a polynomial number of variables and constraints that enables us to consider an exponential number of stochastic scenarios and solve the problem optimally or with a very low optimality gap in a reasonable time. Through these enhancements, our research offers a more flexible, and comprehensive solution to appointment scheduling challenges and sets a new benchmark for personalized patient care and clinic efficiency.

Another important feature studied in this work is that patients can inform the clinic about their availability or unavailability by responding to reminders. This feature enables clinics to explore new strategies for mitigating the adverse effects of no-shows on patient waiting times and provider idle time. We further evaluate the cost-effectiveness and potential benefits of implementing the proposed methodology in clinical settings. We show that the generated schedules reduce total costs by 34\% on average by incorporating patient-dependent service times, and 12\% and 67\% by considering patient-and-time-dependent unpunctuality and no-shows, respectively. Furthermore, we show that considering incentives to encourage patients to respond to reminders and inform the clinic about their no-show has the potential to reduce the total costs further by 23\%. Our research, in addition to the technical contributions, also provides several managerial insights that help healthcare providers improve their services and stay competitive in today's patient-centered healthcare environment. This method will assist clinics in increasing patient satisfaction, which will potentially lead to lower patient turnover and higher patient retention rates which are essential for healthcare providers in the competitive market.

The summary of our contributions are as follows:

\begin{itemize}
    \item We propose a novel stochastic programming model for appointment scheduling under uncertainty, incorporating patient-dependent service times and patient-and-time-dependent unpunctuality and no-shows.
    \item We integrate patient responses to reminders to model their availability and mitigate the impact of no-shows.
    \item We provide extensive computational results and managerial insights to enhance patient satisfaction and clinic performance.
\end{itemize}

The paper is structured as follows: In Section \ref{AS:LR} we review the literature and identify research gaps. In Section \ref{AS:PD}  we define the problem and introduce a novel stochastic programming model for appointment scheduling considering uncertain service times, no-shows, and unpunctuality. In Section \ref{AS:CR}, we present computational results and extensive experiments to evaluate the efficiency of our model. This section also studies the impacts of various factors on clinic efficiency. In Section \ref{AS:CC}  we conclude the paper and provide future research avenues.

\section{Literature Review}\label{AS:LR}
\noindent Our work builds on four key research streams to address challenges related to no-shows, unpunctuality, and variable service times. These streams include: (i) stochastic programming, (ii) dynamic programming, (iii) queuing theory, and (iv) robust optimization. The literature on appointment scheduling in operations management include a wide range of methodologies and approaches with various uncertain parameters. For a detailed literature review, as well as insights into future research avenues, we refer the reader to \citet{Jiang2019Omega} and \citet{Shehadeh2021Stochastic}. In this review, we mention studies that are particularly relevant to our research focus or represent the latest developments in the field.

In the context of stochastic programming, most researchers have proposed two-stage stochastic programming models, mainly within the Sample Average Approximation framework. These studies typically address uncertain parameters including no-shows, random service times, and unpunctuality by incorporating known probability distributions. \hll{For instance, \citet{Jiang2019Omega} study outpatient clinic visits and consider uncertainty in service times and unpunctuality through probability distributions that are identical for all patients. \citet{Zhan2021Home} address in-home service visits, in which providers travel to customers, and handle stochastic service times by generating 1,000 scenarios.} \hll{\citet{Erdogan2013Dynamic} study a generic single-server outpatient setting with no-shows. They propose a model that incorporates no-shows and focuses on assigning appointments to patients dynamically over the planning horizon.} They generate 10,000 random scenarios to show model's efficiency. \hll{\citet{Shehadeh2021Stochastic} have extended this stream to multi-stage stochastic programming for single-provider outpatient procedure visits, considering patient-dependent arrival times and service times. The authors use a Monte Carlo approach to generate 200 scenarios and evaluate the performance of the model.} An exception in this category is the work of \citet{LaGanga2012Overbooking}\hll{, who propose a heuristic for general outpatient clinic sessions that obtains near-optimal overbooked schedules to minimize the effects of no-shows.}

In another area of research, several researchers have proposed dynamic programming models to tackle appointment scheduling problems. These studies mainly employ discrete-time Markov chain processes and often rely on approximation algorithms or heuristics for solutions. \hll{For instance, \citet{Liu2010Dynamic} utilized a Markov Decision Process (MDP) to model the advance booking of outpatient clinic visits with no-shows and cancellations, while \citet{Feldman2014Appointment} developed both static and dynamic models for outpatient visits booked over a multi-day horizon that account for patient preferences and no-shows.} In the static model, different subsets of appointment days are offered to the patient while ignoring the current state of booked appointments. On the other hand, the dynamic model uses the current state of booked appointments to offer the patient a set of appointment days. To solve the problem, they use an approximation approach due to the complexity of the dynamic model. \hll{\citet{Soltani2019Appointment} address stochastic service times and patient no-shows in multi-provider outpatient clinics. The objective function focuses on minimizing patient waiting time, and provider idle time and overtime. They use a machine learning approach to capture patterns and guide a heuristic algorithm toward near-optimal solutions. \citet{Diamant2018} study multi-appointment specialty care programs, motivated by a bariatric surgery assessment clinic, in which patients may miss scheduled visits, and use an approximate dynamic programming algorithm to solve the model.}

Some research in the literature has used queuing theory to model appointment scheduling. \hll{\citet{Liu2016Optimal} proposes an M/M/1/K queuing model for choosing the appointment scheduling window of outpatient visits, with no-show probabilities that differ across categories of patients. \citet{Luo2019Optimization} propose an M/M/1/N queuing model for a hospital outpatient clinic to determine the optimal appointment scheduling window that minimizes no-shows. \citet{zacharias2020} study a generic single-provider outpatient clinic session with time-dependent no-shows, unpunctuality, and stochastic service times.} The authors propose a single-server queuing model to determine the clinic’s workload over time and prove that the objective function is supermodular.

Some works in the literature have developed distributionally robust optimization models to address the stochastic nature of appointment scheduling, mainly focusing on enhancing robustness against worst-case scenarios. These models assume that while some probability distributions characterize the behavior of uncertain parameters, these distributions are unknown and belong to an uncertainty set. \hll{For instance, \citet{Mandelbaum2020Data} study chemotherapy infusion appointments in a cancer center and address uncertain service durations and punctuality using a bootstrapping procedure based on observed data. In another research, \citet{Kong2020} consider outpatient clinic visits with random service times and time-dependent no-shows and aim to minimize worst-case total expected costs. \citet{Jiang2017Integer} study generic single-server appointment systems, such as outpatient clinics and surgical suites, and mainly focus on no-shows considering worst-case expectation and conditional value-at-risk costs.} They use a completely positive decomposition algorithm to solve the robust model by incorporating approximation heuristics during the solution process.

\hll{Closest to our work are studies that pursue an individualized approach toward appointment scheduling. Regarding service times, most models treat durations as identically distributed across patients \citep{Hassin2008, Kaandorp2007Optimal,Liu2016Optimal}, including in primary care office visits where consultation lengths are known to vary widely \citep{TaiSeale2007Time}. Coarse classifications, such as inpatient versus outpatient demand in hospital diagnostic services \citep{Sickinger2009Performance} or new versus returning patients in primary care clinics \citep{Cayirli2008}, capture only part of this heterogeneity because they ignore clinically relevant factors such as chronic conditions and visit reasons. Using data on office-based physician visits, \cite{Salzarulo2016Beyond} showed that schedules built on patient-dependent service times are more accurate. However, they assume identical attendance behavior for all patients and time slots, and their solution methodology is an approximation rather than an exact approach. \cite{Li2019I} estimate individual no-show probabilities with a Bayesian nested logit model for outpatient visits in high-throughput clinics and evaluate the resulting overbooking policies by simulation, without modeling heterogeneous service times or unpunctuality. Similarly, \cite{Kong2020}, discussed above, consider random service times and time-dependent no-shows for outpatient clinic visits, yet their no-show probabilities do not vary across individual patients. In summary, none of these studies jointly capture patient-dependent service times together with patient-and-time-dependent unpunctuality and no-shows within an exact optimization framework.}

The main reason that makes current models in appointment scheduling less efficient is that they often fail to consider patient-and-time-specific factors due to modeling complexity \citep{Li2019I, Kong2020}. To address this gap, our research introduces an innovative stochastic programming approach to the problem. This novel approach enables us to solve the problem with an exponential number of scenarios for the first time and simultaneously consider patient-dependent stochastic service times, and patient-and-time-dependent unpunctuality and no-shows.

\section{Problem Definition and Formulation}\label{AS:PD}

\noindent We consider an appointment scheduling problem where we have to assign appointment times to a certain number of patients within the working hours of the clinic. \hll{Consistent with the outpatient consultation settings described in Section~1, we consider a single provider, such as a primary care physician or a specialist, who serves a sequence of booked patients during one clinic session.} Some of the main challenges in this problem are as follows:

\begin{enumerate}
\item  Service durations are both stochastic and patient-dependent. This variability in service durations reflects the diversity in patients' health conditions. In our study, we explore a broad setting where service durations are allowed to follow diverse, non-identical independent probability distributions without being confined to well-known distributions.

\item  Patient arrival times are stochastic and patient-and-time-dependent that reflect the uncertainty associated with patient unpunctuality. Patient-dependency means that, for each patient, we have a different set of probability distributions to represent his/her unpunctuality. Also, time-dependency refers to the case that the amount of lateness/earliness depends on the allocated appointment time. We consider this assumption to incorporate the effect of traffic in rush hours on patient unpunctuality. This setting allows for greater modeling flexibility by considering patient-and-time-dependency simultaneously.

\item  In some cases, patients may not show up for their appointment. As suggested by \citet{Kong2020}, we assume that the probability of no-shows is time-dependent. Our model advances this assumption by incorporating both time-dependent and patient-dependent probabilities for no-shows, thereby enriching the existing framework with a more detailed understanding of patient behavior.
\end{enumerate}

The primary goal of this research is to minimize the weighted sum of healthcare provider idle times, patient waiting times, and the penalties corresponding to health provider overtime. In the following, initially, we present a simplified model that mainly focuses on the stochastic patient-dependent service durations. This simplification allows us to establish a fundamental understanding of the scheduling dynamics and the presented model without the added complexities of patient unpunctuality and no-shows. Building upon the simplified model, we progressively integrate stochastic patient unpunctuality and no-shows in Section \ref{AS:Sec32}. In Section \ref{AS:Sec33}, we further refine our model to incorporate the impact of appointment reminders and patient notifications on no-show probabilities. This enhancement captures the interactive dynamics between clinic communication strategies and patient behavior. It also offers insights into how proactive patient notifications in response to reminders can influence scheduling efficiency and resource allocation.

\subsection{The Basic Stochastic Appointment Scheduling Model}

\noindent We formulate the problem by a novel stochastic programming approach, called state variable model, that results in a closed-form mixed integer programming model \citep{doulabi2022state}. The main advantage of this formulation is its ability to handle an exponential number of scenarios efficiently using a polynomial number of variables and constraints. This feature is particularly beneficial as it facilitates the solution of large-scale instances with an exponential number of scenarios within a computationally feasible time. To present the proposed model, we define the sets, variables, and parameters as follows:

\noindent
Sets:
\vspace{10 pt}

\begin{tabular}{p{4pt} p{3pt} p{15.1cm}}
\noindent $I$ & : & The ordered set of patients, i.e., $I=\{1,2,\dots ,|I|\}$. Also, we suppose $I_0=I\cup \{0\}$ where "0" denotes a dummy patient with a service time equal to 0 that is scheduled before the first patient at the beginning of the clinic time.  \\[30pt]  
$T$ & : & The set of possible appointment times that are multiples of a small time unit (e.g., 2 minutes) denoted by $\theta$, resulting in $T=\{0,\theta ,2\theta ,3\theta, \dots ,L+O\}$, where $L$ and $O$ are defined as parameters. We suppose that all stochastic service and arrival times are multiples of \textit{$\theta$}. \\[30pt] 
$T_i$ & : & The set of possible appointment times for patient$\ i$. We have
\[
T_i = \begin{cases}
    \{0\}, & i=1, \\
    \left\{t \in T \mid t \ge \sum_{i' \in I : i' < i} t^{min}_{i'}\right\}, & i \in I \setminus \{1\}, \\
    \{.\}, & i = \left|I\right|+1.
\end{cases}
\]
$T_1=\{0\}$ shows that we always offer the first available time to patient 1. In computation of $T_i$ for $i\in I$, we consider the minimum service time of previous patients denoted by $t^{min}_{i'}$. Also, $i=\left|I\right|+1$ denotes a dummy patient that is always scheduled after the last patient at a dummy time $\mathrm{\{}$.$\mathrm{\}}$. \\[30pt] 
$F_i$ & : & The set of possible finish times for patient$\ i$. We have 
\[
F_i = \begin{cases}
    \left\{t \in T \mid t \ge \sum_{i' \in I : i' \le i} t^{min}_{i'}\right\}, & i \neq 0, \\
    \{0\}, & i = 0.
\end{cases}
\]

\end{tabular}

\vspace{10 pt}

\noindent Parameters:

\vspace{10 pt}

\begin{tabular}{p{22pt}p{3pt}p{14.5cm}}
\noindent $L$& : &  The available clinic time.  \\ [10 pt]
$O$& : &  The maximum available overtime of the clinic.  \\  [10 pt]
$t^{min}_i$& : &  The minimum possible service time of patient $i$.  \\  [10 pt]
$c^p_{if_it_{i+1}}$& : &  The waiting cost of patient $i+1$ provided that he/she is given an appointment time of $t_{i+1}$ and the service to patient$\ i$ has finished at the time $f_i$. We have $c^p_{if_it_{i+1}}={\alpha }^p{\left[f_i-t_{i+1}\right]}^+$ where ${\alpha }^p$ is the penalty for a patient waiting for a time unit and ${\left[x\right]}^+$ is equal to $x$ if it is positive and 0 otherwise.\\  [10 pt]
$c^d_{if_it_{i+1}}$& : &  The health provider idle cost while waiting for patient $i+1$ provided that the patient is given an appointment time of $t_{i+1}$ and the service to the previous patient$\ i$ has finished at the time $f_i$. We have $c^d_{if_it_{i+1}}={\alpha }^d{\left[t_{i+1}-f_i\right]}^+$ where ${\alpha }^d$ is the penalty for the health provider waiting for a time unit.  
\end{tabular}

\begin{tabular}{p{22pt}p{3pt}p{14.5cm}}
\noindent $c^o_{if_it_{i+1}}$ & : &  The expected over time cost corresponding to the service to patient $i+1$ provided that the patient is given an appointment time of $t_{i+1}$ and the service to the previous patient$\ i$ has finished at the time $f_i$. We have \newline $c^o_{if_it_{i+1}}={\alpha }^o\left(\ E[{\mathrm{max} \left(L,{\mathrm{max} \left(f_i,t_{i+1}\right)\ }+{\tilde{d}}_{i+1}\right)\ }-\mathrm{max}\mathrm{}(L,f_i)]\right)$ where ${\tilde{d}}_{i+1}$ is a random variable representing the service duration of the patient $i+1$. \\  [10 pt]

\end{tabular}

\vspace{10 pt}

\noindent
Decision Variables:\vspace{6pt}

\begin{tabular}{p{22pt}p{3pt}p{14.5cm}}
\noindent $w_{it_i}$ & : & 1 if we assign time slot $t_i$ to patient$\ i$; 0 otherwise. \\  [10 pt]
$z_{if_it_{i+1}}$ & : & The probability that the service to patient$\ i$ finishes at time $f_i$ (for  $f_i<O+L$) and we have $w_{\left(i+1\right)t_{i+1}}=1$, i.e., the next patient $i+1$ is scheduled at time $t_{i+1}$. If $w_{\left(i+1\right)t_{i+1}}=1$ holds, $z_{if_it_{i+1}}$ will be equal to the above probability. Otherwise, it will be equal to 0. In the case of $\ \ i=\left|I\right|+1$, $t_{i+1}$ can only take the dummy value ``.'' and we ignore the last condition $w_{\left(i+1\right)t_{i+1}}=1$.   \\  [10 pt]
\end{tabular}
\vspace{0 pt}

\noindent
Random Variable:

\begin{tabular}{p{22pt}p{3pt}p{14.5cm}}
\noindent ${\widetilde{d}}_i$ & : & Random service time of patient $i$. \\  [10 pt]
\end{tabular}

\vspace{15 pt}

Using the given notation, we formulate the simplified appointment scheduling problem as (M1).

\begin{eqnarray}
& \hspace{-290pt} \textrm{(M1)} \ & \hspace{-260pt} \min \sum_{i\in I_0}\sum_{f_i\in F_i}\sum_{t_{i+1}\in T_{i+1}}\bigg(c_{if_it_{i+1}}^p+c_{if_it_{i+1}}^d  + c_{if_it_{i+1}}^o\bigg)z_{if_it_{i+1}} \label{eq.1}\\[0pt]
& \hspace{-150pt}  \small \text{Subject to:} & \hspace{-200 pt}  \nonumber \\[-5pt]
& \hspace{3pt} \sum\limits_{t_i\in T_i} w_{it_i}=1 & \hspace{3pt} i\in I \label{eq.2}\\[5pt]
& \hspace{-4pt} \sum\limits_{f_{i-1}\in F_{i-1}} z_{(i-1)f_{i-1}t_i}=w_{it_i} & \hspace{3pt} i\in I,\ t_i\in T_i \label{eq.3}\\[5pt]
& \hspace{-18pt} \sum\limits_{f_{i-1}\in F_{i-1}}\sum\limits_{t_i\in T_i}z_{(i-1)f_{i-1}t_i}\Pr\left(\max\left(t_i,f_{i-1}\right)+{\widetilde{d}}_i=f_i\right) = \sum\limits_{t_{i+1}\in T_{i+1}} z_{if_it_{i+1}} & \hspace{3pt} i\in I,\ f_i\in F_i \label{eq.4}\\[5pt]
& \hspace{-18pt} \sum\limits_{f_{i-1}\in F_{i-1}}\sum\limits_{t_i\in T_i}z_{(i-1)f_{i-1}t_i}\Pr\left(\max\left(t_i,f_{i-1}\right)+{\widetilde{d}}_{i}> L+O\right) = 0 & \hspace{3pt} i\in I \label{eq.5}\\[5pt]
&\hspace{5 pt} w_{it_i}\in\{0,1\}  & \hspace{3pt} i\in I,\ t_i\in T_i \label{eq.6}\\[5pt]
&\hspace{5 pt} 0\le z_{if_it_{i+1}}\le1  & \hspace{3pt} i\in I^0,\ f_i\in F_i,\ t_{i+1}\in T_{i+1} \label{eq.7} 
\end{eqnarray}

Objective function (\ref{eq.1}) minimizes the expected value of the total cost. Constraint (\ref{eq.2}) ensures that we assign exactly a single appointment time to each patient to avoid any missing or duplicated appointments. Constraint (\ref{eq.3}) establishes a relationship between variables $z_{\mathrm{(}i\mathrm{-}\mathrm{1)}f_{i\mathrm{-}\mathrm{1}}t_i}$ and $w_{it_i}$. This constraint implies that if $w_{it_i}$ is equal to 0, then all corresponding variables $z_{\mathrm{(}i\mathrm{-}\mathrm{1)}f_{i\mathrm{-}\mathrm{1}}t_i}$ are equal to 0. This is consistent with the definition of variables$\mathrm{\ }z_{\mathrm{(}i\mathrm{-}\mathrm{1)}f_{i\mathrm{-}\mathrm{1}}t_i}$. When $w_{it_i}$ is equal to 1, the sum of variables $z_{\mathrm{(}i\mathrm{-}\mathrm{1)}f_{i\mathrm{-}\mathrm{1}}t_i}$ on the left-hand side of constraint (\ref{eq.3}) must be equal to 1. This reflects that for a fixed $i\mathrm{\in }I\mathrm{,\ and\ }t_i\mathrm{\in }T_i$, $z_{\mathrm{(}i\mathrm{-}\mathrm{1)}f_{i\mathrm{-}\mathrm{1}}t_i}$ variables represent the probability distribution of the finish time of service to the patient ($i\mathrm{-}\mathrm{1}$). Constraint (\ref{eq.4}) consecutively compute the values of variables $z_{if_it_{i\mathrm{+1}}}$ using conditional probability relations based on variables $z_{\mathrm{(}i\mathrm{-}\mathrm{1)}f_{i\mathrm{-}\mathrm{1}}t_i}$ and the transition probability $\mathrm{Pr(max}\mathrm{}\mathrm{(}t_i,f_{i\mathrm{-}\mathrm{1}}\mathrm{)+}{\tilde{d}}_{i }\mathrm{=}f_i\mathrm{)}$. The right-hand side of this constraint calculates the probability that the appointment of patient $i$ finishes at the time $f_i$. The left-hand side of the constraint (\ref{eq.4})  also computes the same probability by multiplying $z_{\mathrm{(}i\mathrm{-}\mathrm{1)}f_{i\mathrm{-}\mathrm{1}}t_i}$, which represents the probability that the appointment of the patient $(i-1)$ finishes at the time $f_{i\mathrm{-}\mathrm{1}}$ and patient $i$ is scheduled at the time $t_i$,  by the probability $\mathrm{Pr}\mathrm{}\mathrm{(max}\mathrm{}\mathrm{(}t_i,f_{i\mathrm{-}\mathrm{1}}\mathrm{)+}{\tilde{d}}_{i }\mathrm{=}f_i\mathrm{)}$. This represents the probability that the duration of the appointment for patient $i\ $is such that we finish patient $i$ at the time $f_i\mathrm{.\ }$ In modeling constraint (\ref{eq.4}), we assume that prior scheduling decisions do not affect the service time distributions of subsequent patients. Constraint (\ref{eq.5}) ensures that all patients are served before the closing time of the clinic.

\begin{theorem}
\label{thm:model_validity}
Model (M1) is a valid formulation for the appointment scheduling problem.
\end{theorem}

In this theorem, the validity of the model means that probability values represented by state variables are computed correctly, based on the appointment variables. The proof of Theorem~\ref{thm:model_validity} is provided in the supplementary materials. To illustrate how the proposed model works, especially in addressing service time uncertainty, we provide a numerical example in the supplementary materials.

\subsection{Incorporating No-Show and Unpunctuality} \label{AS:Sec32}

\noindent In this section, we enhance our model by incorporating stochastic arrival times and no-shows to further align it with the unpredictable nature of real-world healthcare scheduling scenarios. We introduce the random variable ${\tilde{a}}_{it}$ to represent time perturbations in the arrival of patient \textit{i} from his/her appointment scheduled at time $t$. A negative value of ${\tilde{a}}_{it}$ indicates that the patient arrives earlier than expected, whereas a positive value indicates a delay. Moreover, we define parameter ${\pi }_{it}$ to represent the probability of a no-show for patient \textit{i} at appointment time \textit{t}. These enhancements are crucial for capturing the dynamic and uncertain aspects of patient attendance behavior. We also note that these stochastic parameters take into account the patient-and-time-dependency simultaneously.

To integrate these factors into our state-variable model, we modify the transition probabilities in constraints (\ref{eq.4}) and (\ref{eq.5}). The revised probabilities account for variations in arrival times, as denoted by ${\tilde{a}}_{it}$, and the probability of patient no-shows, as represented by ${\pi }_{it}$. Therefore, the updated transition probability in constraint (\ref{eq.4}), which now incorporates the stochastic nature of patient arrivals and no-show probabilities, is formulated as:
		\begin{align}
		&\mathrm{(1-}{\pi}_{it}\mathrm{)Pr}\mathrm{}\mathrm{(max}\mathrm{}\mathrm{(}t_i\mathrm{+}{\tilde{a}}_{it},f_{i\mathrm{-}\mathrm{1}}\mathrm{)+}{\tilde{d}}_i\mathrm{=}f_i\mathrm{)}+{\pi}_{it}{\mathrm{Pr} \left(\left({\mathrm{max} \left(t_i,f_{i\mathrm{-}\mathrm{1}}\right)\ }\right)\mathrm{=}f_i\right)\ }.  \label{eq.8}
		\end{align}
Similarly, the revised probability for constraint (\ref{eq.5}) is formulated as:
		\begin{align}
		& \mathrm{(1-}{\pi }_{it}\mathrm{)Pr}\mathrm{}\mathrm{(max}\mathrm{}\mathrm{(}t_i\mathrm{+}{\tilde{a}}_{it},f_{i\mathrm{-}\mathrm{1}}\mathrm{)+}{\tilde{d}}_i\mathrm{> } L + O\mathrm{)} . \label{eq.9}
		\end{align}

In (\ref{eq.8}), $(1-\pi_{it}) \mathrm{Pr}\left(\mathrm{max}\left(t_i + \tilde{a}_{it}, f_{i-1}\right) + \tilde{d}_i = f_i\right)$ and ${\pi}_{it}  \mathrm{Pr}\left(\mathrm{max} \left(t_i, f_{i-1}\right) = f_i\right)$ represent the transition probabilities for cases where the patient shows up and does not show up, respectively. In both scenarios, the transitions result in a finish time of $f_i$ for the end of service to patient $i$. 
Similarly, in (\ref{eq.9}), the transition probabilities account for scenarios where the end of service to patient $i$ extends beyond the overtime period. To avoid accounting for the waiting time of no-show patients, we need to incorporate the effect of no-shows in the waiting coefficient \( c^p_{if_it_{i+1}} \) by redefining it as  
$c^p_{if_it_{i+1}}=(1-\pi_{it}){\alpha }^p{\left[f_i-t_{i+1}\right]}^+$. These modifications in the transition probabilities and the waiting coefficient ensure that our model accurately reflects the stochastic nature of patient arrival times and no-shows when determining the appointment start and finish times. This enhances the model's applicability and reliability in practical healthcare scheduling cases.

A key assumption in our proposed model is the availability of patient-dependent and patient-and-time-dependent probability distributions for service times, arrival deviations, and no-show behavior. While this assumption enables a rich and flexible modeling framework, in practice such distributions must be estimated from historical data that may be limited or noisy. In the remainder of this subsection, we discuss how these distributions can be learned and validated in a real clinical setting. Modern clinics routinely collect operational data through electronic health record (EHR) systems and appointment scheduling platforms. These systems provide time-stamped data on scheduled appointment times, actual arrival times, service start and end times, and attendance outcomes. From these records, one can construct patient-level histories capturing service durations, arrival deviations, and no-show occurrences. In addition, relevant covariates such as patient demographics, visit type (e.g., new versus returning), time of day, day of week, and external factors such as weather or traffic conditions can be incorporated to enrich the dataset.

When the amount of data per patient or per patient–time pair is limited, direct empirical estimation of distributions may lead to high variance and unreliable estimates. To address this challenge, data-driven predictive models based on historical electronic medical records can be employed to estimate patient-specific and time-dependent uncertainty parameters. For instance, logistic regression and probabilistic models have been widely used to estimate no-show probabilities by incorporating patient demographics, appointment characteristics, and historical attendance behavior (\citealp{alaeddini2011probabilistic, torres2015risk}). More recently, machine learning approaches such as random forests, gradient boosting, and neural networks have been used to jointly predict no-shows and service durations, capturing complex nonlinear relationships and improving prediction accuracy (\citealp{srinivas2021consultation}). Similarly, patient arrival behavior and punctuality can be modeled using classification and regression-based machine learning techniques that account for patient history, appointment timing, and external factors (\citealp{srinivas2020machine, abu2022multi}). In addition, feature selection and multivariate analysis methods can be used to identify the most significant predictors and reduce model complexity, thereby improving robustness in data-sparse settings (\citealp{leiva2025predictive}). To ensure the reliability of the estimated distributions, these predictive models are typically validated using train-test splits or cross-validation techniques, with performance assessed through standard metrics such as AUC, calibration measures, or prediction error.

Overall, although the exact distributions are not directly observable, a combination of data-driven statistical methods and careful validation enables the practical implementation of the proposed modeling framework in real clinical settings.

\subsection{Incorporating Reminders} \label{AS:Sec33}

\noindent In this section, we augment our model to encompass a realistic scenario in healthcare scheduling that is when patients notifying the clinic in advance if they are unable to attend their appointment. This scenario is particularly relevant for clinics that utilize reminder systems and prompt patients to confirm or cancel their appointments.

To model this new feature, we introduce a probability ${\pi }^{info}$ which represents the likelihood of patients notifying the clinic about their no-show in response to reminders. Using this probability, we can evaluate the effectiveness of the reminder system in prompting patient responses regarding their appointment attendance.

The incorporation of ${\pi }^{info}$ necessitates modifications to the transition probabilities in constraints (\ref{eq.4}) and (\ref{eq.5}). These adjustments are essential to accurately represent the dynamics of the system when patients proactively communicate their attendance or absence. By considering stochastic service times, arrivals, no-shows, and the probability of responding to reminders, we formulate the new transition probability for constraint (\ref{eq.4}) as:
\begin{equation}
\begin{split}
\pi_{it}\left[\pi^{info}\mathbf{1}\left(f_{i-1}=f_i\right) + (1-\pi^{info})\Pr\left(\max(t_i,f_{i-1})=f_i\right)\right] \\
+(1-\pi_{it})\Pr\left(\max(t_i+\tilde{a}_{it},f_{i-1})+\tilde{d}_i=f_i\right).
\end{split} 
\label{eq.10}
\end{equation}

In (\ref{eq.10}), $\pi_{it}\pi^{info}\mathbf{1}\left(f_{i-1}=f_i\right)$ represents the scenario where the patient does not show up but informs the clinic of their no-show, while $\pi_{it}(1-\pi^{info})\Pr\left(\max(t_i,f_{i-1})=f_i\right)$ addresses the case where the patient fails to show up without prior notification. 

It is important to note that the validity of this enhanced model originates from the validity of the original model as demonstrated in supplementary materials. The primary difference in the enhanced model is in the updated transition probabilities which now accommodate the added detail of patient responses to reminders.

\section{Computational experiments}\label{AS:CR}

\noindent Our computational analysis aims to evaluate the performance of the proposed state-variable model through a series of targeted questions, each designed to assess a critical aspect of healthcare appointment scheduling. These computational analyses will aim at answering the following questions:

\begin{enumerate}
\item  Does our proposed model demonstrate computational efficiency?

\item What is the impact of incorporating different patient health states on the appointment scheduling efficiency when addressing patient-dependent service times?

\item  What is the impact of patient-and-time-dependent unpunctuality on the appointment scheduling efficiency?

\item What is the impact of considering patient-and-time-dependent no-shows on the appointment scheduling efficiency?

\item What is the impact of incorporating reminders in encouraging patients to disclose their no-show behavior on the appointment scheduling efficiency?

\item What is the impact of increased patient show-up rates, driven by incentives, on the appointment scheduling efficiency?

\end{enumerate}

To design our computational experiments, we consider the three key factors on clinic operational efficiency including service times, stochastic arrivals, and no-shows. As it will be discussed in Section \ref{sec:4.1.1}, we categorize service times into four levels and both stochastic arrivals and no-shows into five levels to consider the independent and simultaneous effect of patient and time dependency in these factors. Our computational experiments include the following analyses to answer the above questions:

\begin{enumerate}
\item  To answer Question 1, first, we evaluate the computational efficiency of the model by solving several instances using CPLEX. This assessment will provide insights into the practicality and scalability of the model in various settings.

\item  To answer Questions 2, 3, and 4, we separately explore the effects of incorporating patient-dependent service times, and patient-and-time-dependent stochastic arrival times, and no-shows into the scheduling model. 

\item  To answer Question 5, we assess the influence of patient information-sharing through responding to reminders on appointment scheduling. This includes studying how encouraging patients to proactively communicate their attendance intentions, particularly through no-show disclosures in response to reminders, affect scheduling efficiency.
 
\item To address Question 6, we introduce a scaling parameter to adjust patient show-up probabilities. This parameter is essential for analyzing how changes in attendance rates, influenced by incentives or negative clinic experiences, impact the appointment scheduling efficiency.

\end{enumerate}

In the following, first we explain the instance generation and the different settings for HS, SA, and NS factors in Section \ref{sec:4.1}. Then in Section \ref{sec:4.2}, we provide the results of the computational experiments and provide several managerial insights.

\subsection{Instance generation} \label{sec:4.1}

\noindent To generate instances, we consider session lengths of 150, 210, and 270 minutes to reflect diversity in clinical practices as in \citep{Salzarulo2016Beyond,Klassen2009Improving}. Moreover, we consider a predetermined overtime allowance of 60 minutes. Our model incorporates three cost parameters including ${\alpha }^d$, ${\alpha }^p$, and ${\alpha }^o$ which represent the cost per minute for healthcare provider idleness, patient wait time, and overtime, respectively. In our research, the cost ratio factor $({\alpha }^d\mathrm{/}{\alpha }^p)$ represents the relative importance of healthcare provider time to patient waiting time. Following \citet{Klassen2009Improving} and \citet{Salzarulo2016Beyond}, we explore the cost ratio across three distinct levels including 1, 5, and 10, to understand its impact on operational efficiency. Moreover, the overtime cost is set to 1.5 times higher than the idle time cost \citep{Salzarulo2016Beyond}. We investigate the above-mentioned settings in instances with 10, 12, and 14 patients to assess the effectiveness of the model across various clinical settings. This approach enables us to thoroughly evaluate the interplay of these diverse factors and their collective impact on the efficiency of clinical appointment scheduling. 

Following \citet{Klassen2009Improving}, we consider no-show rates ranging from 0\% to 30\% that are generated using a uniform distribution. To generate realistic values for unpunctuality, we assume an average early arrival time of 10 minutes. The data is derived from several normal distributions with means of -5, -10, and -15 minutes and a standard deviation of 1.7  as suggested by \citet{Cayirli2008} and \citet{Salzarulo2016Beyond}.

\subsubsection{Service times} \label{sec:4.1.1}

\noindent 
We generate Set 1 of our instance to address Question 2 presented at the beginning of Section 4. In this set, as in \citet{Salzarulo2016Beyond}, we model the service times as probability functions of patients' health states. In the following, we consider four levels of information on patients' health states. Each level progressively incorporates more complexity regarding patient-specific details, as outlined in Table \ref{tab:HSTable} and detailed as follows:

\noindent \textbf{HS0 (Level 1):} This baseline level does not incorporate any specific health information and assumes that the service duration for all patients follows a general distribution representative of the average patient population.

\noindent \textbf{HS1 (Level 2):} At this level, patients are classified as either returning or new. We consider a truncated normal distribution with mean $\mu_i$ and standard deviation $\sigma_i$ for patient $i$, using three-sigma bounds on each side, with the parameters estimated by \citet{Salzarulo2016Beyond} as follows:

\vspace{-20pt}

\begin{align}
\ln\left(\mu_i\right) &= 2.329RP + 2.247NP, \label{eq.12}\\
\sigma_i &= 0.658. \label{eq.13}
\end{align}

\noindent \textbf{HS2 (Level 3):} Here, we incorporate more patient-specific information and categorize patients into four health states including Low Health, Moderate Health, High Health, and Excellent Health. In addition, we consider factors such as prescription medication (PM), physical examination (PE), total chronic conditions (TC), and unspecified visit reasons (NR). The log-transformed mean examination time is computed as follows \citep{Salzarulo2016Beyond}:
\vspace{-1pt}
\begin{gather}
\ln\left(\mu_i\right) = 3.088H^{Low}_i + 2.430H^{Mid}_i + 2.341H^{High}_i + 2.184H^{Perf}_i \notag\\
+ 0.147TC - 0.106NR + 0.315PE - 0.102PM, \label{eq.14}
\end{gather}

with a standard error of ${\sigma }_i=0.644$.

\noindent \textbf{HS3 (Level 4):} This advanced level includes the main effects of HS2 and additional two-way interactions. At this level, we calculate the log transformation of i${}^{th}$ patient's mean examination duration and the standard deviation as follows \citep{Salzarulo2016Beyond}:
\vspace{-1pt}
\begin{align}
{\mathrm{ln} \left({\mu }_i\right)\ }=3.112H^{Low}_i+2.462H^{Mid}_i+2.388H^{High}_i+2.245H^{Perf}_i+0.145TC- \notag\\
\quad 0.116NR+0.265PE-0.108PM-0.105NP-0.369BC+0.785NP(BC), \label{eq.15} 
\end{align}
\vspace{-1.5cm}
\begin{align}
{\sigma}_i &= 0.642. \label{eq.16}
\end{align}

Table \ref{tab:HSTable} presents a concise summary of the key factors considered at each health state level. This Table outlines the presence (+) or absence (-) of each factor across HS0, HS1, HS2, and HS3 \citep{Salzarulo2016Beyond}.

\begin{table}[htbp] 
\centering
\caption{Key factors at each health state level \citep{Salzarulo2016Beyond}}
\label{tab:HSTable}
\fontsize{10}{10}\selectfont
\begin{tabular}{|p{2.5in}|>{\centering\arraybackslash}p{0.5in}|>{\centering\arraybackslash}p{0.5in}|>{\centering\arraybackslash}p{0.5in}|>{\centering\arraybackslash}p{0.5in}|}
\hline
\multirow{2}{*}{\parbox[c][2\baselineskip][c]{2.5in}{\centering Key Factors}} & \multicolumn{4}{c|}{Health State} \\
\cline{2-5}
 & \centering HS0 & \centering HS1 & \centering HS2 & \centering\arraybackslash HS3 \\ 
\hline
Return Patients (RP) & - & + & - & - \\ 
\hline
New patients (NP) & - & + & + & - \\ 
\hline
Low Health- (H${}_{i}^{Low}$) & - & - & + & + \\ 
\hline
Moderate Health- (H${}_{i}^{Mid}$) & - & - & + & + \\ 
\hline
High Health- (H${}_{i}^{Hi}$) & - & - & + & + \\ 
\hline
Excellent Health- (H${}_{i}^{Perf}$) & - & - & + & + \\ 
\hline
Total Chronic (TC) & - & - & + & + \\ 
\hline
No Reason (NR) & - & - & + & + \\ 
\hline
Physical Examination (PE) & - & - & + & + \\ 
\hline
Prescription Medication (PM) & - & - & + & + \\ 
\hline
Birth Control (BC) & - & - & - & + \\ 
\hline
New Patient $\times$ Birth Control NP(BC) & - & - & - & + \\ 
\hline
\end{tabular}
\end{table}

To assess the impact of varying levels of health state information levels on appointment scheduling, we generate 10 instances for each combination of clinic time, cost ratio, and the number of patients.  This approach yields a substantial dataset comprising 1080 unique instances that offers a rich basis for in-depth analysis. This comprehensive approach allows us to perform a detailed investigation of how health state information affects appointment scheduling in clinical settings.

\subsubsection{Patient Unpunctuality}
\noindent 
We generate Set 2 of our instance to address Question 3 presented at the beginning of Section 4. In this set, we examine different levels of patient unpunctuality categorized as SA0, SA1, SA2, SA3, and SA4 where each represents varying degrees of information about patient unpunctuality.

\noindent \textbf{SA0 (Level 1):} This baseline level assumes all patients are punctual and arrive on time.

\noindent \textbf{SA1 (Level 2):} Here, we introduce a single unpunctuality distribution for all patients that is uniform across various time slots. This setting represents the case that the clinic collects unpunctuality data regardless of the patient information and the appointment time.

\noindent \textbf{SA2 (Level 3):} We consider different unpunctuality distributions for different time slots. We note that the distribution considered for each time slot is independent of the patient information, and it applies to the time slot regardless of the patient information. This setting represents the case that the clinic collects the unpunctuality data considering the appointment time regardless of the patient information.

\noindent \textbf{SA3 (Level 4):} At this level, we consider different patient-dependent unpunctuality distributions that apply to all time slots. Essentially, it acknowledges the variance in patient arrival patterns but assumes these patterns are consistent across various times of the day. This setting models a scenario where the clinic gathers unpunctuality data unique to each patient but does not differentiate the data based on appointment times.

\noindent \textbf{SA4 (Level 5):} In the most detailed level, SA4, we introduce patient-and-time-dependent unpunctuality distributions. This setup represents the case where the clinic collects detailed unpunctuality data considering both the patient behavior and the appointment times.

In the following, we first explain the formulas for SA4 in which the unpunctuality of patient $i$ is patient-and-time-dependent. Then, we provide information about other levels that are simplified versions of the fifth level. For SA4, we model patient unpunctuality using a truncated normal distribution with a mean ${\mu}^\prime_{it}$ and standard deviation ${\sigma }_{arrival}$. For this purpose, we present the following equation:
\begin{gather}
{\mu}^\prime_{it}={\mu}^\prime_i+\beta\left(\overline{\mu^\prime }\left(\frac{{\alpha }_t}{\overline{\alpha }}\right)-\overline{\mu^\prime}\right), \label{eq.17}
\end{gather}

\noindent where ${\mu}^\prime_i$ stands for the mean unpunctuality of patient $i$ during regular traffic. Besides, the coefficient $\beta $ is a control parameter to increase or decrease the effect of traffic. Coefficient ${\alpha }_t$ is the traffic index that reflects the relative traffic volume at time slot $t$. The average traffic index $\overline{\alpha }$ and average unpunctuality $\overline{\mu^\prime}$  are calculated over all time slots and patients, respectively, as follows. 
\begin{gather}
\overline{\alpha }=\frac{\sum_{t\in T}{{\alpha }_t\mathrm{\ }}}{|T|}. \label{eq.18}\\
\overline{\mu^\prime}=\frac{\sum_{i\in I}{{\mu }^\prime_i\mathrm{\ }}}{|I|}. \label{eq.19}
\end{gather}

In Eq.(\ref{eq.18}), for the values of ${\alpha }_{t}$ we have used the traffic data of arterial roads in urban areas of Ohio in 2016 \citep{OhioDOT2023}. Also, in Eq.(\ref{eq.19}), we generate patient-specific mean unpunctuality ${\mu }^\prime_i$ using a uniform distribution in the range of [-25, -5] minutes. This results in $\mathrm{|}I\mathrm{|\times |}T\mathrm{|}$ distributions for ${\mu }^\prime_{it}$ in SA4. In all these truncated normal distributions, we considered the domain of [${\mu }^\prime_{it}-3{\sigma }_{it}$, ${\mu }^\prime_{it}+3{\sigma }_{it}$] for the random parameter where ${\sigma }_{it}$ is the standard deviation of the distribution.

To construct the unpunctuality models for SA1, SA2, and SA3, we average {}${\mu }^\prime_{it}$ over different dimensions. For SA3, we calculate the average of the parameter ${\mu }^\prime_{it}$ over time slots yielding a unique truncated normal distribution for each patient. For SA2, we average ${\mu }^\prime_{it}$ over patients for each time slot which yields a truncated normal distribution for each time slot. Finally, for SA1, {}parameter ${\mu }^\prime_{it}$ is averaged over both patients and time slots which creates a single distribution applicable to all patients and time slots.

To assess the impact of varying levels of detail about patient unpunctuality on appointment scheduling, we generate 10 instances for each combination of clinic time, cost ratio, and the number of patients that results in a total of 1350 instances. Such an extensive evaluation is crucial to thoroughly investigate how different levels of unpunctuality information affect appointment scheduling. This will enhance our understanding of the dynamics involved in managing patient arrivals in various healthcare settings.

\subsubsection{No-shows}
\noindent 
We generate Set 3 of our instances to address Question 3 presented at the beginning of Section 4, where we have integrated no-shows as a main factor. We consider five distinct levels including NS0, NS1, NS2, NS3, and NS4 to represent varying degrees of information about patient no-shows.

\noindent \textbf{NS0 (Level 1):} At this level, we consider a basic scenario with no specific no-show data in which all patients are assumed to attend their appointments.

\noindent \textbf{NS1 (Level 2):} Here a single distribution for no-shows applicable to all patients across all time slots. The distribution is valid for all time slots and all patients as we do not have time and patient dependency. This level models a scenario where a clinic collects no-show data regardless of the patient and the appointment time.

\noindent \textbf{NS2 (Level 3):} This level introduces time-dependent no-show distributions. We consider a specific distribution for each timeslot that applies to all patients assigned to that timeslot. This setting represents a case in which clinics collect no-show data based on appointment times while ignoring individual patient characteristics.

\noindent \textbf{NS3 (Level 4):} Here, we focus on patient-dependent no-show distributions that are consistent across different time slots. This approach represents clinics that collect individual patient no-show data regardless of their appointment times.

\noindent \textbf{NS4 (Level 5):} This is the most detailed level, where no-show probabilities are tailored to individual patients and specific time slots. This setting depicts a clinic collecting comprehensive no-show data considering both patient behavior and appointment times.

In the following, we first explain the process to generate NS4, and then we provide information about NS3, NS2, NS1, and NS0 levels. For NS4, the no-show probability for each patient\textit{ i} at time slot $t$ is set to ${\pi }_{it}\mathrm{=}{\pi }^{\mathrm{'}}_i\mathrm{+}{\pi }^{\mathrm{''}}_t$, where parameter ${\pi }^{\mathrm{''}}_t$ is a time-dependent no-show parameter and parameter ${\pi }^{\mathrm{'}}_i$ is a patient-dependent constant that enables ${\pi }_{it}$ to take varying values for different patients at a particular time slot $t$. For ${\pi }^{\mathrm{''}}_t$, we refer to Figure 1-a in \citet{Kong2020} for time-dependent no-show probabilities. Besides, we uniformly generated ${\pi }^{\mathrm{'}}_i$ in the interval [-0.3, 0.3]. 

To generate NS3, we calculate the average of ${\pi }_{it}$ over different time slots for each patient and obtained a distinct truncated normal distribution for each patient. For NS2, we compute the average of ${\pi }_{it}$ across patients for each time slot which yields a unique truncated normal distribution for each time slot. Furthermore, for NS1, we calculate the average of ${\pi }_{it}$ across both patients and time slots and determine one truncated normal distribution for all patients and all time slots. 

To assess the impact of varying no-show levels on appointment scheduling, we generate 10 instances for each combination of clinic time, cost ratio, and the number of patients that results in a total of 1350 instances.

\subsection{Computational results} \label{sec:4.2}

\noindent In this section, we report the computational results to answer questions presented at the beginning of Section \ref{AS:CR}. To solve the proposed model, we use CPLEX 12.10 on a PC with an AMD Rome 7532 @ 2.40 GHz 256M cache L3 CPU and 16 GB of RAM. We considered a time limit of 24 hours for each instance. Remarkably, the majority of instances were resolved well before this limit that shows the effectiveness of the model in handling complex scheduling problems. This performance indicates not only the computational viability of the proposed approach for real-world application but also its potential for scalability to larger or more complex instances in healthcare scheduling.

\subsubsection{Tractability of the model}
\noindent 
In the following, we perform a comprehensive assessment of our proposed appointment scheduling model and focus particularly on its performance and computational efficiency. We use the CPLEX solver and extensively test the model across various settings. Table \ref{tab:CPLEXTABLE} presents the summary of the computational results of the CPLEX on 270 instances. In this table, the first three columns indicate the settings of instances including the number of patients, clinic time, and cost ratio. In the next columns, we report 1) the number of the stochastic scenarios captured by our model, 2) the number of variables, 3) the number of constraints, 4) the number of nodes exploited in the branch-and-bound tree, 5) lower bound, 6) upper bound, 7) time: computational time of the model in seconds, 8) gap(\%): which is computed by gap = $\frac{100\left(UB\ -\ LB\right)}{LB}$ waiting cost, 10) idle cost, and 11) overtime cost. We note that in all these instances we considered HS, NS, and SA at their highest level of detail that provide in-depth information about the health state, no-show, and unpunctuality of the patients. This means that the unpunctuality and no-show probabilities follow patient-and-time-dependent distributions, and service time durations follow patient-dependent distributions.

Table \ref{tab:CPLEXTABLE} reveals that the proposed model efficiently solves appointment scheduling problems with an exponential number of stochastic scenarios within a reasonable time limit. This represents a significant accomplishment, especially considering the complexity and scale of the addressed scheduling problem, which involve patient-and-time-dependent unpunctuality and no-shows, along with patient-dependent service times. The average optimality gap is 0.2$\%$ with the majority of instances achieving optimal solutions within the allocated time limit. The average number of variables and constraints are 171,323 and 12,020, respectively, which are relatively low considering the exponential number of stochastic scenarios captured by the proposed model. This highlights the efficiency and innovation of our approach.

Figure \ref{fig:finishtimeprobability} presents a visual representation of the calculated finish time probabilities as determined by our proposed model for a given instance with 14 patients. Based on Figure \ref{fig:finishtimeprobability}, we observe that the probability of finish time of the appointment for last patients becomes more uncertain compared to those of first patients. This is mainly due to the fact that the uncertainty in no-show, unpunctuality, and service durations of the first patients affect the start time of the appointment for later patients.

\begin{landscape}

\begin{table}
\caption{Computational results of the CPLEX in solving the proposed state-variable model.}\label{tab:CPLEXTABLE}
{
\fontsize{8}{10}\selectfont
 \renewcommand{\arraystretch}{1.3} 
    \begin{tabular}{@{}@{}*{14}{c}@{}}

    \hline
No. of   patients & Clinic time       & Cost ratio          & No. of scenarios
                          & V        & C     & N           & LB                      & UB                      & Time (sec)             & Gap (\%)               & Waiting Cost           & Idle Cost              & Overtime Cost          \\\hline
10                & 150               & 1                   & 9.14E+4050                  & 83031                   & 6880                   & 1551                   & 19.96                   & 19.96                   & 833                    & 0.00                  & 6.31                   & 13.65                  & 0.00                   \\
                  &                   & 5                   & 9.14E+4050                  & 83031                   & 6880                   & 2326                   & 51.08                   & 51.08                   & 856                    & 0.00                  & 25.41                  & 25.66                  & 0.01                   \\
                  &                   & 10                  & 9.14E+4050                  & 83031                   & 6880                   & 2367                   & 69.57                   & 69.57                   & 894                    & 0.00                  & 41.53                  & 28.02                  & 0.01                   \\
                  & 210               & 1                   & 8.04E+5198                  & 135903                  & 8824                   & 1987                   & 19.31                   & 19.31                   & 1554                   & 0.00                  & 5.59                   & 13.72                  & 0.00                   \\
                  &                   & 5                   & 8.04E+5198                  & 135903                  & 8824                   & 2354                   & 49.74                   & 49.74                   & 1454                   & 0.00                  & 24.49                  & 25.24                  & 0.00                   \\
                  &                   & 10                  & 8.04E+5198                  & 135903                  & 8824                   & 3070                   & 68.00                   & 68.00                   & 1735                   & 0.00                  & 41.16                  & 26.84                  & 0.00                   \\
                  & 270               & 1                   & 8.28E+6349                  & 201735                  & 10768                  & 1843                   & 19.94                   & 19.94                   & 1834                   & 0.00                  & 6.23                   & 13.71                  & 0.00                   \\
                  &                   & 5                   & 8.28E+6349                  & 201735                  & 10768                  & 2736                   & 51.05                   & 51.05                   & 2351                   & 0.00                  & 25.70                  & 25.35                  & 0.00                   \\
                  &                   & 10                  & 8.28E+6349                  & 201735                  & 10768                  & 2350                   & 69.54                   & 69.54                   & 2040                   & 0.00                  & 41.53                  & 28.01                  & 0.00                   \\\hline
12                & 150               & 1                   & 9.06E+4859                  & 101435                  & 9259                   & 8420                   & 31.45                   & 31.45                   & 9336                   & 0.00                  & 10.10                  & 20.97                  & 0.38                   \\
                  &                   & 5                   & 9.06E+4859                  & 101435                  & 9259                   & 13114                  & 81.28                   & 81.28                   & 11098                  & 0.00                  & 40.13                  & 39.70                  & 1.45                   \\
                  &                   & 10                  & 9.06E+4859                  & 101435                  & 9259                   & 10622                  & 112.21                  & 112.21                  & 8629                   & 0.00                  & 64.24                  & 44.92                  & 3.04                   \\
                  & 210               & 1                   & 6.95E+6239                  & 166043                  & 11875                  & 7638                   & 30.04                   & 30.04                   & 15247                  & 0.00                  & 9.72                   & 20.32                  & 0.00                   \\
                  &                   & 5                   & 6.95E+6239                  & 166043                  & 11875                  & 11051                  & 78.30                   & 78.30                   & 17283                  & 0.00                  & 39.42                  & 38.88                  & 0.00                   \\
                  &                   & 10                  & 6.95E+6239                  & 166043                  & 11875                  & 15599                  & 106.94                  & 106.94                  & 17161                  & 0.00                  & 63.18                  & 43.76                  & 0.00                   \\
                  & 270               & 1                   & 5.59E+7618                  & 246491                  & 14491                  & 7399                   & 31.19                   & 31.19                   & 21833                  & 0.00                  & 10.16                  & 21.03                  & 0.00                   \\
                  &                   & 5                   & 5.59E+7618                  & 246491                  & 14491                  & 6071                   & 80.43                   & 80.43                   & 18945                  & 0.00                  & 40.39                  & 40.04                  & 0.00                   \\
                  &                   & 10                  & 5.59E+7618                  & 246491                  & 14491                  & 7860                   & 110.34                  & 110.34                  & 17537                  & 0.00                  & 62.61                  & 47.74                  & 0.00                   \\\hline
14                & 150               & 1                   & 7.17E+5669                  & 119839                  & 11978                  & 14762                  & 41.12                   & 41.12                   & 21304                  & 0.00                  & 14.91                  & 19.82                  & 6.40                   \\
                  &                   & 5                   & 7.17E+5669                  & 119839                  & 11978                  & 22654                  & 107.94                  & 107.94                  & 30993                  & 0.00                  & 47.52                  & 40.32                  & 20.11                  \\
   &    &   10 &   7.17E+5669 &   119839 &   11978 &   39844 &   158.07 &   158.68 &   43079 &   0.49 &   72.48 &   47.74 &   38.47 \\
& 210 & 1  & 9.69E+7279 & 196183 & 15362 & 15203 & 32.03  & 32.15  & 42442 & 0.37 & 10.78 & 21.37 & 0.00  \\ 
                  &                   & 5                   & 9.69E+7279                  & 196183                  & 15362                  & 21997                  & 83.97                   & 84.94                   & 47453                  & 1.13                  & 41.02                  & 43.91                  & 0.00                   \\
                  &                   & 10                  & 9.69E+7279                  & 196183                  & 15362                  & 36605                  & 113.69                  & 116.76                  & 62017                  & 2.90                  & 65.43                  & 51.33                  & 0.00                   \\
                  & 270               & 1                   & 1.97E+8890                  & 291247                  & 18746                  & 4502                   & 34.58                   & 34.58                   & 18916                  & 0.00                  & 13.36                  & 21.22                  & 0.00                   \\
                  &                   & 5                   & 1.97E+8890                  & 291247                  & 18746                  & 7313                   & 88.29                   & 88.29                   & 22508                  & 0.00                  & 45.79                  & 42.50                  & 0.00                   \\
                  &                   & 10                  & 1.97E+8890                  & 291247                  & 18746                  & 16904                  & 120.30                  & 121.05                  & 43601                  & 0.60                  & 71.94                  & 49.11                  & 0.00                   \\\hline
\multicolumn{3}{l}{Average}                           &                             & 171323                  & 12020                  & 10672                  & 68.90                   & 69.11                   & 17886                  & 0.20                  & 34.86                  & 31.66                  & 2.59    \\ \hline              
    \end{tabular}
}

\end{table}
{Notes: V: No. of variables, C: No. of constraints, N: No. of nodes.}
\end{landscape}

\renewcommand{\arraystretch}{1} 

\noindent

\begin{figure}
\centering
\includegraphics[width=0.8\textwidth]{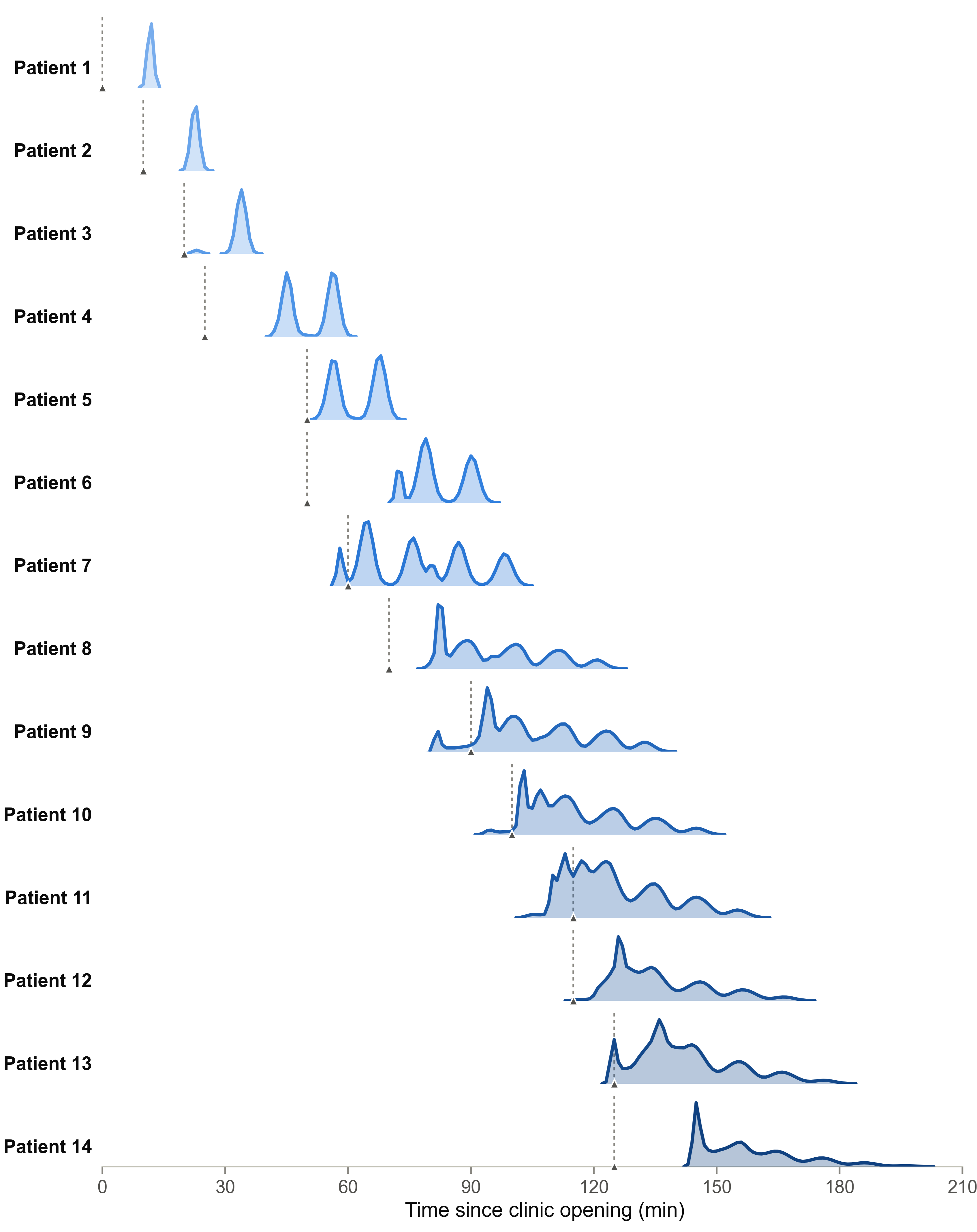}
\caption{Calculated finish time probabilities.}
\label{fig:finishtimeprobability}
\end{figure}

\subsubsection{Comparison with the Baseline Algorithm}
\label{subsec:bailey_baseline}

\noindent
In this section, we compare the proposed approach with the classical Bailey appointment rule. In the Bailey framework, the first two patients are scheduled at time zero, and the remaining patients are assigned according to a prespecified spacing rule. We considered two variants. In the first variant, denoted as the \emph{Bailey interval} rule, a fixed interval is used between consecutive patients after the first two. In the second variant, denoted as the \emph{Modified Bailey} rule, alternating intervals are used. For example, under the $(5,15)$ rule, the first two patients are scheduled at time zero, the third patient is scheduled 5 minutes later, the fourth patient 15 minutes later, and this pattern is repeated throughout the session. The corresponding results are reported in Tables~\ref{tab:bailey_interval}, \ref{tab:bailey_combo_1}, and \ref{tab:bailey_combo_2}. In all tables, the reported percentage gaps are computed relative to the optimal schedule obtained by the proposed model under the same stochastic setting.

The results show that the proposed model consistently outperforms all Bailey-based schedules in terms of total expected cost. For the fixed-interval Bailey policy, the best-performing rule is the 10-minute interval, yet it still yields an average increase of 97.29\% in total cost relative to the proposed model. By contrast, the 5-minute interval performs very poorly, with an average total-cost increase of 419.65\%. This deterioration is mainly driven by a dramatic increase in patient waiting cost, which rises by 465.12\% on average, even though provider idle cost is reduced by 47.73\%. This behavior is intuitive: when the schedule is made too dense, the provider is utilized more aggressively, but the resulting congestion shifts the burden to patients through excessive waiting. At the other extreme, larger fixed intervals such as 15 and 20 minutes substantially reduce waiting cost on average, by 9.35\% and 27.95\%, respectively, but these gains come at the expense of a very large increase in provider idle cost, equal to 254.28\% and 549.10\%, respectively. Consequently, the average total-cost increases remain very high at 245.29\% and 521.61\%. These findings indicate that fixed Bailey rules cannot adequately balance patient waiting and provider idleness when service times, no-shows, and unpunctuality are explicitly modeled.

A similar pattern is observed for the Modified Bailey rules. Among the tested combinations, the $(5,15)$ rule provides the best overall performance, but its average total cost is still 116.84\% higher than that of the proposed model. Interestingly, this rule produces only a small average change in idle cost ($-7.24\%$) and overtime cost ($-0.19\%$), but it still increases waiting cost by 124.28\%, which is sufficient to make the overall schedule clearly inferior. The $(10,15)$ and $(5,20)$ rules reduce the waiting-cost gap relative to the most congested schedules, but they do so by creating substantially larger idle and overtime penalties. In particular, the $(10,15)$ rule increases idle cost and overtime cost by 97.55\% and 8.70\% on average, respectively, which results in a 132.59\% increase in total cost. Likewise, the $(5,20)$ rule leads to a 143.02\% increase in total cost. The $(5,10)$ rule performs the worst among the Modified policies, with an average total-cost increase of 229.69\%, again due to excessive patient waiting.

\sk{The variation across cost-ratio settings can be explained by the interaction between the rigidity of Bailey-type schedules and the relative weights assigned to waiting and idle costs. When the cost ratio is low, patient waiting has a stronger influence on the total objective. In this setting, dense Bailey schedules perform poorly because they reduce provider idle time by scheduling patients aggressively, but this creates congestion and substantially increases patient waiting. As the cost ratio increases, provider idle time becomes more expensive relative to patient waiting. In such cases, Bailey schedules may appear relatively better because their dense structure keeps the provider more continuously occupied. However, this improvement is only partial because Bailey-type rules still cannot adapt appointment times to patient-dependent service durations, patient-and-time-dependent arrival behavior, and patient-and-time-dependent no-show probabilities. Therefore, the performance gap varies across cost-ratio settings because each cost ratio emphasizes a different side of the waiting-idle tradeoff, whereas the proposed model explicitly optimizes this tradeoff under the full stochastic structure.}

Overall, the comparison highlights an important limitation of Bailey-type policies in the setting studied in this paper. These rules are simple and easy to implement, but they impose a rigid appointment structure that cannot adapt to patient-dependent service-time distributions and patient-and-time-dependent no-show and unpunctuality behavior. As a result, Bailey schedules tend to overemphasize one operational objective at the expense of another: dense schedules reduce idle time but create excessive patient waiting, whereas sparse schedules reduce waiting but generate large amounts of provider idle time. In contrast, the proposed model explicitly optimizes this trade-off while accounting for the full stochastic structure of the problem. Therefore, the computational evidence suggests that the benefit of the proposed approach is not only theoretical but also practically significant when compared with a widely used baseline scheduling rule.

\begin{table}[htp]
\caption{Bailey method comparison with the optimal appointment schedule.}
\label{tab:bailey_interval}
\resizebox{\textwidth}{!}{%
\begin{tabular}{lllllll}
\hline
Bailey interval & Clinic time & Cost ratio & $\Delta_{\text{Idle Cost}}$ (\%) & $\Delta_{\text{Waiting Cost}}$ (\%) & $\Delta_{\text{Overtime Cost}}$ (\%) & $\Delta_{\text{Total Cost}}$ (\%) \\ \hline

        5        & 150  & 1           &     -48.66 &         761.79 &          -0.96 &     712.18 \\
                 &      & 5           &     -39.45 &         259.49 &           8.85 &     228.92 \\
                 &      & 10          &     -32.17 &         161.35 &          12.16 &     141.36 \\
                 & 210  & 1           &     -67.50 &         988.28 &           0.03 &     920.81 \\
                 &      & 5           &     -49.89 &         335.12 &           0.06 &     285.29 \\
                 &      & 10          &     -40.54 &         219.80 &           0.08 &     179.34 \\
                 & 270  & 1           &     -63.16 &         928.09 &           0.00 &     864.93 \\
                 &      & 5           &     -48.08 &         319.83 &           0.00 &     271.75 \\
                 &      & 10          &     -40.09 &         212.36 &           0.00 &     172.27 \\
        \hline
        Average  &      &             &     -47.73 &         465.12 &           2.25 &     419.65 \\ \hline

        10       & 150  & 1           &     -26.20 &         186.08 &          -5.17 &     154.72 \\
                 &      & 5           &       3.36 &          40.95 &           0.89 &      45.21 \\
                 &      & 10          &      28.05 &          11.29 &           1.68 &      41.02 \\
                 & 210  & 1           &     -42.71 &         262.84 &           0.01 &     220.14 \\
                 &      & 5           &      -3.15 &          61.68 &           0.01 &      58.55 \\
                 &      & 10          &      27.21 &          21.91 &           0.01 &      49.13 \\
                 & 270  & 1           &     -39.05 &         243.48 &           0.00 &     204.43 \\
                 &      & 5           &      -1.62 &          56.32 &           0.00 &      54.69 \\
                 &      & 10          &      27.96 &          19.73 &           0.00 &      47.69 \\
        \hline
        Average  &      &             &      -2.90 &         100.48 &          -0.28 &      97.29 \\ \hline

        15       & 210  & 1           &      97.48 &          38.60 &           0.38 &     136.46 \\
                 &      & 5           &     261.04 &         -22.84 &           0.72 &     238.92 \\
                 &      & 10          &     409.85 &         -39.20 &           1.03 &     371.68 \\
                 & 270  & 1           &      95.12 &          31.93 &           0.00 &     127.05 \\
                 &      & 5           &     256.56 &         -24.96 &           0.00 &     231.60 \\
                 &      & 10          &     405.62 &         -39.63 &           0.00 &     365.99 \\
        \hline
        Average  &      &             &     254.28 &          -9.35 &           0.35 &     245.29 \\ \hline

        20       & 270  & 1           &     250.14 &           1.11 &           0.24 &     251.48 \\
                 &      & 5           &     554.97 &         -36.74 &           0.47 &     518.70 \\
                 &      & 10          &     842.18 &         -48.22 &           0.68 &     794.64 \\
        \hline
        Average  &      &             &     549.10 &         -27.95 &           0.46 &     521.61 \\ \hline

\end{tabular}
}
\end{table}

\begin{table}[htp]
\caption{Modified Bailey method comparison with the optimal appointment schedule.}
\label{tab:bailey_combo_1}
\resizebox{\textwidth}{!}{%
\begin{tabular}{lllllll}
\hline
Modified Bailey & Clinic time & Cost ratio & $\Delta_{\text{Idle Cost}}$ (\%) & $\Delta_{\text{Waiting Cost}}$ (\%) & $\Delta_{\text{Overtime Cost}}$ (\%) & $\Delta_{\text{Total Cost}}$ (\%) \\ \hline

        (5,10)   & 150  & 1           &     -46.26 &         455.11 &          -3.49 &     405.37 \\
                 &      & 5           &     -34.83 &         143.02 &           4.04 &     112.24 \\
                 &      & 10          &     -25.66 &          81.23 &           5.82 &      61.40 \\
                 & 210  & 1           &     -64.88 &         604.97 &           0.02 &     540.10 \\
                 &      & 5           &     -44.95 &         190.64 &           0.03 &     145.72 \\
                 &      & 10          &     -33.36 &         115.22 &           0.04 &      81.90 \\
                 & 270  & 1           &     -60.61 &         566.10 &           0.00 &     505.50 \\
                 &      & 5           &     -43.14 &         180.47 &           0.00 &     137.33 \\
                 &      & 10          &     -32.83 &         110.47 &           0.00 &      77.63 \\
        \hline
        Average  &      &             &     -42.95 &         271.91 &           0.72 &     229.69 \\ \hline

        (5,15)   & 150  & 1           &     -28.31 &         223.49 &          -5.01 &     190.17 \\
                 &      & 5           &      -0.67 &          55.13 &           1.18 &      55.64 \\
                 &      & 10          &      22.43 &          21.01 &           2.06 &      45.50 \\
                 & 210  & 1           &     -45.16 &         310.34 &           0.01 &     265.19 \\
                 &      & 5           &      -7.78 &          79.58 &           0.02 &      71.82 \\
                 &      & 10          &      20.51 &          34.85 &           0.02 &      55.38 \\
                 & 270  & 1           &     -41.39 &         288.25 &           0.00 &     246.86 \\
                 &      & 5           &      -6.15 &          73.54 &           0.00 &      67.39 \\
                 &      & 10          &      21.34 &          32.31 &           0.00 &      53.64 \\
        \hline
        Average  &      &             &      -7.24 &         124.28 &          -0.19 &     116.84 \\ \hline

        (5,20)   & 150  & 1           &      20.35 &          99.56 &           7.99 &     127.91 \\
                 &      & 5           &      91.79 &           8.19 &          26.14 &     126.12 \\
                 &      & 10          &     151.42 &         -10.86 &          36.40 &     176.95 \\
                 & 210  & 1           &      10.83 &         146.15 &           0.00 &     156.98 \\
                 &      & 5           &      97.65 &          17.68 &           0.01 &     115.34 \\
                 &      & 10          &     173.25 &          -9.91 &           0.00 &     163.34 \\
                 & 270  & 1           &      12.61 &         133.86 &           0.00 &     146.47 \\
                 &      & 5           &      97.73 &          14.17 &           0.00 &     111.89 \\
                 &      & 10          &     173.29 &         -11.06 &           0.00 &     162.23 \\
        \hline
        Average  &      &             &      92.10 &          43.09 &           7.84 &     143.02 \\ \hline
\end{tabular}
}
\end{table}

\begin{table}[htp]
\caption{Modified Bailey method comparison with the optimal appointment schedule (continued).}
\label{tab:bailey_combo_2}
\resizebox{\textwidth}{!}{%
\begin{tabular}{lllllll}
\hline
Modified Bailey & Clinic time & Cost ratio & $\Delta_{\text{Idle Cost}}$ (\%) & $\Delta_{\text{Waiting Cost}}$ (\%) & $\Delta_{\text{Overtime Cost}}$ (\%) & $\Delta_{\text{Total Cost}}$ (\%) \\ \hline

        (10,15)  & 150  & 1           &      22.91 &          73.37 &           9.39 &     105.67 \\
                 &      & 5           &      96.66 &          -1.73 &          28.84 &     123.77 \\
                 &      & 10          &     158.14 &         -17.65 &          40.08 &     180.57 \\
                 & 210  & 1           &      13.99 &         112.78 &           0.00 &     126.78 \\
                 &      & 5           &     103.63 &           5.10 &           0.01 &     108.74 \\
                 &      & 10          &     181.89 &         -19.01 &           0.00 &     162.88 \\
                 & 270  & 1           &      15.58 &         102.17 &           0.00 &     117.75 \\
                 &      & 5           &     103.48 &           1.98 &           0.00 &     105.45 \\
                 &      & 10          &     181.69 &         -19.96 &           0.00 &     161.73 \\
        \hline
        Average  &      &             &      97.55 &          26.34 &           8.70 &     132.59 \\ \hline

        (10,20)  & 210  & 1           &      93.81 &          55.37 &           0.30 &     149.47 \\
                 &      & 5           &     254.10 &         -16.53 &           0.57 &     238.13 \\
                 &      & 10          &     399.81 &         -34.64 &           0.82 &     365.99 \\
                 & 270  & 1           &      91.65 &          47.91 &           0.00 &     139.56 \\
                 &      & 5           &     249.86 &         -18.83 &           0.00 &     231.03 \\
                 &      & 10          &     395.82 &         -35.15 &           0.00 &     360.67 \\
        \hline
        Average  &      &             &     247.51 &          -0.31 &           0.28 &     247.48 \\ \hline

        (15,20)  & 210  & 1           &     177.72 &          22.12 &          22.11 &     221.94 \\
                 &      & 5           &     412.36 &         -29.03 &          41.85 &     425.18 \\
                 &      & 10          &     629.01 &         -43.66 &          60.53 &     645.87 \\
                 & 270  & 1           &     171.03 &          16.03 &           0.00 &     187.06 \\
                 &      & 5           &     402.67 &         -31.03 &           0.00 &     371.64 \\
                 &      & 10          &     619.37 &         -44.05 &           0.00 &     575.31 \\
        \hline
        Average  &      &             &     402.03 &         -18.27 &          20.75 &     404.50 \\ \hline

\end{tabular}
}
\end{table}

\subsubsection{Analysis of Health States (HS)} \label{subsection 4.2}
\noindent 
In Table \ref{tab:HSsimulations}, we thoroughly examine the impact of varying levels of Health State (HS) on the appointment scheduling problem. In this table, for all instances, we have considered the highest levels of stochastic arrival (SA4) and no-shows (NS4) for all patients. This means we have patient-and-time-dependent unpunctuality and no-show probability distributions. For instances with the setting of HS3, we have reported the results obtained from the proposed model directly in the table. However, the solutions of instances with other health settings (i.e., HS0, HS1, and HS2) must be evaluated in the reference setting of ``HS3-SA4-NS4'' that represents the real-world situation. This approach aligns with our objective outlined in Question 2 of Section \ref{AS:CR} that aims to assess the significance of incorporating detailed patient health information into scheduling effectiveness.

In order to simulate these solutions (HS0, HS1, and HS2) in the context of the most comprehensive health state setting (HS3-SA4-NS4), we first solve the proposed model with ``(HS0, HS1, or HS2)-SA4-NS4'' and save their appointment schedule presented by $w_{it_i}$ variables. Then, we fix the $w_{it_i}$ variables and resolve the model with the reference setting of ``HS3-SA4-NS4'' to assess (or equivalently simulate) the solutions obtained in ``(HS0, HS1, or HS2)-SA4-NS4''.

Our results indicate that as we move from HS0 to more detailed levels of HS (HS1, HS2, and HS3), there is a significant decrease in the total cost, waiting cost, and overtime cost for a given instance. This trend highlights the substantial benefits of incorporating a comprehensive range of health related factors into the scheduling model. By integrating more detailed and patient-specific health information, the model benefits from more accurate probability distributions for service times and subsequently generates more efficient schedules.

When comparing HS0 with HS1, and HS2 with HS3, we notice only marginal improvements in reducing total costs. This suggests that while advancing from a moderate to a more detailed understanding of patient health states does contribute to efficiency, the gains are relatively slight. Given the costs associated with collecting and integrating more detailed health state data, stakeholders must carefully consider whether the slight improvements justify the additional efforts and used resources. This is especially pertinent in settings where resources are limited, and the incremental cost savings might not offset the investment required for gathering comprehensive health data.

\subsubsection{Analysis of Stochastic Arrivals (SA)} \label{subsection 4.3}
\noindent 
In Table \ref{tab:SAsimulations}, we report the results of computational experiments performed to assess the impact of different levels of Stochastic Arrivals (SA) on the efficiency of the appointment scheduling problem. These findings are key in addressing Question 3, as outlined at the beginning of Section 
\ref{AS:CR}. This analysis is instrumental in understanding the extent to which stochastic patient arrival times impact the overall effectiveness and optimization of healthcare scheduling and provide valuable insights for enhancing operational efficiency in clinical settings.

In this table, for all instances, we have considered the highest levels of Health State (HS3) and No-Shows (NS4) for all patients. For instances, we have reported the results obtained from the proposed model in ``HS3-SA4-NS4'' setting directly in the table. However, the solutions of instances with other stochastic arrival settings (i.e., SA0, SA1, SA2, and SA3) must be evaluated in the reference setting of ``HS3-SA4-NS4'', as discussed in the previous subsection. This is because we aim at assessing the importance of including different levels of patients' stochastic arrivals in the effectiveness of appointment scheduling. 

The outcomes show that SA4, which combines both patient-and-time-dependent unpunctuality factors, significantly outperforms other stochastic arrival settings. This indicates that, by integrating the highest level of detail in unpunctuality, the model allocates appointment slots more effectively that leads to reduced costs and enhanced operational efficiency. Settings SA2 and SA3 benefit from modeling two different aspects of uncertainty in patients' arrival times, and therefore none of them is expected to necessarily dominate the other one in all instances. However, interestingly, our results indicate that in most instances SA2 dominates SA3 in terms of the total cost. This observation is reasonable considering that typical traffic patterns are expected to be the main factor affecting the punctuality of patients. However, it is conceivable that patient demographics might emerge as a significant factor influencing scheduling outcomes on a case-by-case basis. Also, our results indicate that SA2 and SA3 outperform SA1, which is a simplified version of both settings considering only one probability distribution for all patients and all time slots. Finally, as expected, we observe SA1 is superior to SA0 that assumes that all patients are punctual. 

\begin{landscape}
\begin{table}

\caption{Computational results for analysis of health states.}
\label{tab:HSsimulations}
{
\fontsize{8}{10}\selectfont
 \renewcommand{\arraystretch}{1.3} 
\begin{tabular}{@{}@{}*{23}{c}@{}}

        \hline
        \multicolumn{3}{c}{Instance Info} & & \multicolumn{4}{c}{Total Cost} & & \multicolumn{4}{c}{Waiting Cost} & & \multicolumn{4}{c}{Idle Cost} & & \multicolumn{4}{c}{Overtime Cost} \\ \cline{1-3} \cline{5-8} \cline{10-13} \cline{15-18} \cline{20-23} \down
        NP & CT & CR & & HS0 & HS1 & HS2 & HS3 & & HS0 & HS1 & HS2 & HS3 & & HS0 & HS1 & HS2 & HS3 & & HS0 & HS1 & HS2 & HS3 \\ \hline
        10 & 150 & 1 &                             & 46.63                   & 45.61                   & 20.48                   & 19.96                   &                      & 41.15                   & 40.08                   & 5.95                    & 6.31                    &                      & 5.48                    & 5.52                    & 14.53                   & 13.65                   &                      & 0                       & 0                       & 0                       & 0                       \\
                       &                         & 5&                              & 79.89                   & 79                      & 51.91                   & 51.08                   &                      & 68.83                   & 68.99                   & 24.59                   & 25.41                   &                      & 11.06                   & 10                      & 27.32                   & 25.66                   &                      & 0.01                    & 0.01                    & 0.01                    & 0.01                    \\
\hline
    \end{tabular}
}

\end{table}
{Notes: NP: No. of patients, CT: Clinic Time, CR: Cost Ratio.}
\end{landscape}
\renewcommand{\arraystretch}{1} 

\begin{landscape}
\begin{table}
\caption{Computational results for analysis of stochastic arrivals.}
\label{tab:SAsimulations}
{
\fontsize{7.3}{10}\selectfont
 \renewcommand{\arraystretch}{1.3} 
    \begin{tabular}{@{}@{}*{27}{c}@{}}

\hline
                     \multicolumn{3}{c}{Instance info} & \multicolumn{1}{c}{}         & \multicolumn{5}{c}{Total Cost}                                                                                                  & \multicolumn{1}{c}{} & \multicolumn{5}{c}{Waiting Cost}                                                                                                & \multicolumn{1}{c}{} & \multicolumn{5}{c}{Idle Cost}                                                                                                   & \multicolumn{1}{c}{} & \multicolumn{5}{c}{Overtime Cost}  \\ \cline{1-3}\cline{5-9} \cline{11-15} \cline{17-21} \cline{23-27}
NP        & CT            & CR & & \multicolumn{1}{c}{SA0} & \multicolumn{1}{c}{SA1} & \multicolumn{1}{c}{SA2} & \multicolumn{1}{c}{SA3} & \multicolumn{1}{c}{SA4} & \multicolumn{1}{c}{} & \multicolumn{1}{c}{SA0} & \multicolumn{1}{c}{SA1} & \multicolumn{1}{c}{SA2} & \multicolumn{1}{c}{SA3} & \multicolumn{1}{c}{SA4} & \multicolumn{1}{c}{} & \multicolumn{1}{c}{SA0} & \multicolumn{1}{c}{SA1} & \multicolumn{1}{c}{SA2} & \multicolumn{1}{c}{SA3} & \multicolumn{1}{c}{SA4} & \multicolumn{1}{c}{} & \multicolumn{1}{c}{SA0} & \multicolumn{1}{c}{SA1} & \multicolumn{1}{c}{SA2} & \multicolumn{1}{c}{SA3} & \multicolumn{1}{c}{SA4} \\ \hline
\multicolumn{1}{r}{10} & \multicolumn{1}{r}{150} & 1&                              & 22.6                    & 21.24                   & 20.51                   & 20.99                   & 19.96                   &                      & 6.16                    & 5.63                    & 6.24                    & 5.38                    & 6.31                    &                      & 16.44                   & 15.62                   & 14.27                   & 15.61                   & 13.65                   &                      & 0                       & 0                       & 0                       & 0                       & 0                       \\
\hline
\end{tabular}
}

\end{table}
{Notes: NP: No. of patients, CT: Clinic Time, CR: Cost Ratio.}
\end{landscape}
\renewcommand{\arraystretch}{1} 

\subsubsection{Analysis of No-Shows (NS)}
\noindent 
In Table \ref{tab:NSsimulations}, we have reported the computational results for assessing the impact of different levels of no-show (NS) on the efficiency of appointment scheduling problem. These findings contribute to addressing research Question 4 as outlined at the beginning of Section \ref{AS:CR}.

In this table, for all instances, we have considered the highest levels of Health State (HS3) and Stochastic Arrivals (SA4) for all patients. As discussed in subsections \ref{subsection 4.2} and \ref{subsection 4.3}, we have evaluated the solutions of instances with NS0, NS1, NS2, and NS3 in the reference setting of ``HS3-SA4-NS4''. Also, the results of NS4 are directly reported as obtained from the proposed model. 

The results indicate that NS4 significantly outperforms other no-show settings thanks to combining patient-and-time-dependent no-show factors. This outcome suggests that more detailed no-show information enables clinics to schedule appointments more effectively thus maximizing resource utilization. Additionally, similar to the SA analysis, we observe NS2 results in lower total cost than NS3 implying that time dependency plays a more critical role than patient dependency in no-shows. This analysis highlights that while our conclusions are generally applicable, variations might occur on a case-by-case basis, where patient demographics could emerge as a more critical factor in no-shows. 

Notably, NS2 and NS3 demonstrate superior performance compared to NS1, which can be attributed to NS1's lack of consideration for both time-dependent and patient-dependent no-show probabilities. Furthermore, NS1 outperforms NS0 which shows the importance of acknowledging no-shows in the scheduling process.

\subsubsection{Analysis of reminder systems}
\noindent 
In this section, we analyze the impacts of effective communication with patients on the efficiency of the clinic's scheduling system particularly in terms of managing no-shows. We investigate whether sending reminders to patients to inform the clinic about their potential no-shows is economically viable and how it influences overall clinic efficiency. To conduct a thorough investigation, we assess various levels of ${\pi }^{info}$ that represents the likelihood of patients revealing their no-shows in response to reminders. We perform this analysis for ${\pi }^{info}\mathrm{\in }\mathrm{\{}\mathrm{0,\ 0.2,\ 0.4,\ 0.6,\ 0.8}\mathrm{,1}\}$ indicating different scenarios from ``no response to reminders'' to ``absolute certainty in patient communication'' regarding no-shows. By examining this spectrum, we aim to understand the effect of different degrees of patient engagement and communication effectiveness on the scheduling process. The primary aim is to address Research Question 5 outlined at the start of Section \ref{AS:CR} which focus on how different extents of patient responsiveness to reminders influence the overall effectiveness of the scheduling strategy.

We analyze instances with 10, 12, and 14 patients to ensure that our findings are applicable across clinics with different sizes. In all these instances, we consider the highest level of information available for HS, SA, and NS to carry out our analysis. The findings presented in Table \ref{tab:reminder} shed light on the potential enhancements in scheduling efficiency that can be realized through the implementation of successful patient communication strategies.

The results show a 23\% reduction in the average total cost as the probability of patients informing the clinic about potential no-shows (${\pi }^{info}$) increases from 0 to 1. This improvement is achieved mainly thanks to substantial decreases in total waiting and idle costs. Our results show that sending reminders to patients to inform the clinic about their potential no-shows is a highly effective strategy in enhancing the overall efficiency of the healthcare system. Our study also reveals that reminder systems results in more significant savings in total, waiting, and idle costs in clinics with higher cost ratios. The results also indicate that our model is computationally effective in providing practitioners with the opportunity to assess the economic feasibility of establishing different reminder systems such as SMS, emails, and phone calls. 

\subsubsection{Analysis of enhanced show-up rate}
\noindent 
In this section, we analyze the effect of increasing patients' show-up rates in the efficiency of the appointment scheduling systems. The motivation behind this analysis is that clinics can increase patients' show-up rates through incentivization and penalty mechanisms. Thus, we introduce a new parameter $\alpha $ that represents the adjustment in no-show rates as in ${\pi }_{it}:= \left(1-\alpha \right){\pi }_{it}$. The values of $\alpha >0$ reflect having enhanced show-up rates, while $\alpha <0$ represents a decreased likelihood of patient attendance, possibly due to negative past experiences. In our analysis, we consider different values of $\alpha \in \{-0.2,-0.1,\ 0,\ 0.1,0.2\}$. Our approach involves comparing the results of the scheduling model with the ``HS3-SA4-NS4'' setting before and after adjusting the no-show rates. We have provided the results in Table \ref{tab:Show-up}. These results are crucial as they provide a deep understanding of the implications of adjusting patient show-up probabilities. 

The extreme case of enhanced show up, $\alpha =0.2$, results in around 10\% improvement in the average total cost compared to the reference setting of $\alpha =0$. This indicates that, by employing strategies such as incentivization, clinics can better utilize their resources, reduce idle times, and minimize the waiting time for patients. On the other hand, the degraded show up rates in the case of $\alpha =-0.2$ results in an average deterioration of 13.58\%. This observation highlights the criticality of enhancing patient experience through multifaceted approaches, such as cultivating a congenial environment within healthcare settings. By implementing strategies like incentivization, clinics can optimize resource utilization, diminish idle periods, and curtail patient waiting times and thus elevate the overall efficiency and satisfaction in healthcare delivery.

\begin{landscape}
\begin{table}
\caption{Computational results for analysis of no-shows.}
\label{tab:NSsimulations}
{
\fontsize{7.1}{10}\selectfont
 \renewcommand{\arraystretch}{1.3} 
\begin{tabular}{@{}@{}*{27}{c}@{}}

\hline
                       \multicolumn{3}{c}{Instance info} & \multicolumn{1}{c}{}         & \multicolumn{5}{c}{Total Cost}                                                                                                  & \multicolumn{1}{c}{} & \multicolumn{5}{c}{Waiting Cost}                                                                                                & \multicolumn{1}{c}{} & \multicolumn{5}{c}{Idle Cost}                                                                                                   & \multicolumn{1}{c}{} & \multicolumn{4}{c}{Overtime Cost}                                                                     &                         \\ \cline{1-3}\cline{5-9} \cline{11-15} \cline{17-21} \cline{23-27}
NP        & CT            & CR & & \multicolumn{1}{c}{NS0} & \multicolumn{1}{c}{NS1} & \multicolumn{1}{c}{NS2} & \multicolumn{1}{c}{NS3} & \multicolumn{1}{c}{NS4} & \multicolumn{1}{c}{} & \multicolumn{1}{c}{NS0} & \multicolumn{1}{c}{NS1} & \multicolumn{1}{c}{NS2} & \multicolumn{1}{c}{NS3} & \multicolumn{1}{c}{NS4} & \multicolumn{1}{c}{} & \multicolumn{1}{c}{NS0} & \multicolumn{1}{c}{NS1} & \multicolumn{1}{c}{NS2} & \multicolumn{1}{c}{NS3} & \multicolumn{1}{c}{NS4} & \multicolumn{1}{c}{} & \multicolumn{1}{c}{NS0} & \multicolumn{1}{c}{NS1} & \multicolumn{1}{c}{NS2} & \multicolumn{1}{c}{NS3} & \multicolumn{1}{c}{NS4} \\
\hline
\end{tabular}
}

\end{table}
{Notes: NP: No. of patients, CT: Clinic Time, CR: Cost Ratio.}
\end{landscape}

\renewcommand{\arraystretch}{1} 

\begin{landscape}
\begin{table}
\caption{Computational results for analysis of reminder systems.}
\label{tab:reminder}
{
\fontsize{6.7}{10}\selectfont
 \renewcommand{\arraystretch}{1.3} 
    \begin{tabular}{@{}c@{}c@{}c@{}p{0cm}*{5}{c}p{0.3cm}p{-0.1cm}*{5}{c}p{0.3cm}p{-0.1cm}*{5}{c}p{0.3cm}p{-0.1cm}*{6}{c}@{}}

\hline
                \multicolumn{3}{c}{Instance info} & \multicolumn{1}{c}{}         & \multicolumn{6}{c}{Average of Total Cost}          &  & \multicolumn{6}{c}{Average of Waiting Cost}        &  & \multicolumn{6}{c}{Average of Idle Cost}          &  & \multicolumn{6}{c}{Average of Overtime Cost}  \\ \cline{1-3} \cline{5-10} \cline{12-17} \cline{19-24}  \cline{26-31}
NP & CT & CR && 0     & 0.2   & 0.4   & 0.6   & 0.8   & 1     &  & 0     & 0.2   & 0.4   & 0.6   & 0.8   & 1     &  & 0     & 0.2   & 0.4   & 0.6   & 0.8   & 1     &  & 0    & 0.2  & 0.4  & 0.6  & 0.8  & 1    \\ \hline
10              & 150         & 1&          & 19.96 & 19.28 & 18.44 & 17.35 & 15.92 & 13.84 &  & 6.31  & 6.29  & 5.89  & 5.77  & 5.54  & 5.23  &  & 13.65 & 12.99 & 12.55 & 11.58 & 10.38 & 8.61  &  & 0    & 0    & 0    & 0    & 0    & 0    \\
                &             & 5&          & 51.08 & 49.57 & 47.49 & 44.77 & 40.65 & 34.25 &  & 25.41 & 24.71 & 24.48 & 21.97 & 20.84 & 17.52 &  & 25.66 & 24.86 & 23.01 & 22.8  & 19.8  & 16.72 &  & 0.01 & 0.01 & 0.01 & 0.01 & 0.01 & 0.01 \\
                &             & 10&         & 69.57 & 67.65 & 64.98 & 61.34 & 55.94 & 46.61 &  & 41.53 & 40.2  & 38.56 & 34.98 & 32.71 & 27.52 &  & 28.02 & 27.43 & 26.4  & 26.35 & 23.22 & 19.08 &  & 0.01 & 0.01 & 0.01 & 0.01 & 0.01 & 0.01 \\
                & 210         & 1&          & 19.31 & 18.58 & 17.7  & 16.48 & 14.92 & 12.61 &  & 5.59  & 5.51  & 5.77  & 5.48  & 5.06  & 4.38  &  & 13.72 & 13.07 & 11.93 & 11.01 & 9.86  & 8.23  &  & 0    & 0    & 0    & 0    & 0    & 0    \\
                &             & 5 &         & 49.74 & 48.06 & 45.91 & 42.96 & 38.48 & 31.28 &  & 24.49 & 24.02 & 23.49 & 21.33 & 18.91 & 15.82 &  & 25.24 & 24.05 & 22.42 & 21.63 & 19.57 & 15.46 &  & 0    & 0    & 0    & 0    & 0    & 0    \\
\hline
\end{tabular}%
}

\end{table}
{Notes: NP: No. of patients, CT: Clinic Time, CR: Cost Ratio.}
\end{landscape}

\renewcommand{\arraystretch}{1} 

\begin{landscape}
\begin{table}
\caption{Computational results for show-up adjustment coefficient.}
\label{tab:Show-up}
{
\fontsize{7.3}{10}\selectfont
 \renewcommand{\arraystretch}{1.3} 
    \begin{tabular}{@{}@{}*{29}{c}@{}}                
\hline
\multicolumn{3}{c}{} & \multicolumn{1}{c}{}         &  \multicolumn{5}{c}{Average of Total}               & \multicolumn{1}{c}{} & \multicolumn{5}{c}{Average of Waiting}               & \multicolumn{1}{c}{} & \multicolumn{5}{c}{Average of Idle}                  & \multicolumn{1}{c}{} & \multicolumn{5}{c}{Average of Overtime}              \\ \cline{5-9} \cline{11-15} \cline{17-21}  \cline{23-27}
\multicolumn{3}{c}{Instance info} & \multicolumn{1}{c}{}                      & \multicolumn{5}{c}{Show-up adjustment coefficient} & \multicolumn{1}{c}{} & \multicolumn{5}{r}{Show-up   adjustment coefficient} &                      & \multicolumn{5}{r}{Show-up   adjustment coefficient} &                      & \multicolumn{5}{r}{Show-up   adjustment coefficient} \\ \cline{1-3} \cline{5-27} 
NP   & CT            & CR  &     & -0.2     & -0.1     & 0        & 0.1     & 0.2     &                      & -0.2      & -0.1     & 0        & 0.1      & 0.2     &                      & -0.2     & -0.1     & 0        & 0.1      & 0.2      &                      & -0.2      & -0.1      & 0        & 0.1     & 0.2     \\ \hline
\multicolumn{1}{c}{10}                & \multicolumn{1}{c}{150} & 1&                                     & 25.41    & 22.8     & 19.96    & 17.53   & 15.67   &                      & 8.28      & 6.93     & 6.31     & 5.75     & 4.74    &                      & 17.12    & 15.88    & 13.65    & 11.78    & 10.93    &                      & 0         & 0         & 0        & 0       & 0       \\
                                      &                         & 5&                                     & 64.5     & 58.71    & 51.08    & 45.42   & 40.9    &                      & 33.95     & 30.89    & 25.41    & 22.57    & 21.17   &                      & 30.54    & 27.82    & 25.66    & 22.84    & 19.71    &                      & 0         & 0         & 0.01     & 0.01    & 0.02    \\
\hline
\end{tabular}%
}

\end{table}
{Notes: NP: No. of patients, CT: Clinic Time, CR: Cost Ratio.}
\end{landscape}
\renewcommand{\arraystretch}{1} 



\subsubsection{Sensitivity Analysis}
\label{subsec:sensitivity_analysis}

\noindent
In this subsection, we evaluate the robustness of the proposed model with respect to estimation errors in the main uncertainty inputs. We perturb average service times, average arrival deviations, and no show probabilities by 5\%, 10\%, and 15\%. The resulting schedules are evaluated under the most detailed setting with HS3, SA4, and NS4.

Table~\ref{service_time_sensitivity} reports the results for service time perturbations. The proposed model remains stable under moderate changes in average service times. The average total cost changes by only 0.29\% and 0.91\% under 5\% and 10\% perturbations, respectively. Even with a 15\% perturbation, the average total cost increases by 5.09\%. The largest changes occur in waiting cost, while overtime cost remains limited in most cases. These results indicate that the proposed model is robust to moderate errors in estimating service time distributions.

Table~\ref{arrival_times_sensitivity} presents the results for arrival time perturbations. The schedules show a high level of robustness to errors in estimating average arrival deviations. The average total cost changes by only 0.06\%, 0.61\%, and 0.32\% under 5\%, 10\%, and 15\% perturbations, respectively. Across all settings, the effect on overtime cost is negligible. Most of the variation is reflected in waiting and idle costs. This suggests that moderate estimation errors in patient arrival behavior have a limited impact on schedule performance.

\sk{Table~\ref{no_show_sensitivity} reports the results for no show probability perturbations. The average total cost decreases by 0.75\%, 0.67\%, and 0.56\% under 5\%, 10\%, and 15\% perturbations, respectively. This reduction is mainly driven by lower waiting costs, especially when the cost ratio is high. At the same time, idle cost may increase because fewer patients attend their appointments. These results show that the model remains stable under moderate errors in no show probability estimates. They also highlight the operational tradeoff between reducing patient waiting and maintaining provider utilization.}

\begin{table}[htp]
\caption{Service times sensitivity analysis.}
\label{service_time_sensitivity}
\resizebox{\textwidth}{!}{%
\begin{tabular}{lllllll}
\hline
$\Delta$ in average durations & Clinic time & Cost ratio & $\Delta_{\text{Idle Cost}}$ (\%) & $\Delta_{\text{Waiting Cost}}$ (\%) & $\Delta_{\text{Overtime Cost}}$ (\%) & $\Delta_{\text{Total Cost}}$ (\%) \\ \hline

5\%    & 150 & 1           & -0.17                    & 0.53                         & 0.30                  & 0.66               \\
        &                     & 5           & -0.08                    & 0.24                         & 0.20                  & 0.36               \\
        &                     & 10          & -0.02                    & 0.15                         & 0.25                  & 0.38               \\
        & 210 & 1           & -0.04                    & 0.25                         & 0.00                  & 0.21               \\
        &                      & 5           & 0.26                     & -0.20                        & 0.00                  & 0.07               \\
        &                      & 10          & 0.08                     & 0.05                         & 0.01                  & 0.14               \\
        & 270 & 1           & 0.05                     & 0.28                         & 0.00                  & 0.34               \\
        &                      & 5           & -0.05                    & 0.46                         & 0.00                  & 0.41               \\
        &                      & 10          & -0.06                    & 0.12                         & 0.00                  & 0.06               \\ \hline
     Average & & & 0.00                     & 0.21                         & 0.09                  & 0.29               \\ \hline
10\%     & 150 & 1           & 0.28                     & 0.40                         & 0.12                  & 0.80               \\
        &                      & 5           & -0.02                    & 0.69                         & 0.17                  & 0.84               \\
        &                      & 10          & 0.26                     & 0.32                         & 0.57                  & 1.15               \\
        & 210 & 1           & 0.50                     & 0.80                         & 0.00                  & 1.30               \\
        &                      & 5           & 0.71                     & -0.13                        & 0.00                  & 0.58               \\
        &                      & 10          & 0.66                     & 0.05                         & 0.00                  & 0.71               \\
        & 270 & 1           & 0.76                     & 0.82                         & 0.00                  & 1.58               \\
        &                      & 5           & 1.33                     & -1.07                        & 0.00                  & 0.26               \\
        &                      & 10          & 1.10                     & -0.12                        & 0.00                  & 0.98               \\ \hline
 Average & &  & 0.62                     & 0.20                         & 0.10                  & 0.91               \\ \hline
15\%     & 150 & 1           & -1.59                    & 8.53                         & 0.63                  & 7.56               \\
        &                      & 5           & -1.15                    & 4.36                         & 1.62                  & 4.82               \\
        &                      & 10          & -0.73                    & 3.13                         & 2.69                  & 5.11               \\
        & 210 & 1           & -0.24                    & 6.79                         & 0.02                  & 6.57               \\
        &                      & 5           & -0.79                    & 4.64                         & 0.04                  & 3.89               \\
        &                      & 10          & 0.09                     & 3.59                         & 0.07                  & 3.75               \\
        & 270 & 1           & -1.44                    & 8.19                         & 0.00                  & 6.74               \\
        &                      & 5           & -0.37                    & 3.74                         & 0.00                  & 3.36               \\
        &                      & 10          & -0.74                    & 4.72                         & 0.00                  & 3.99               \\ \hline
      Average & & & -0.78                    & 5.30                         & 0.56                  & 5.09  \\   \hline         
\end{tabular}

}

\end{table}

\begin{table}[htp]
\caption{Arrival times sensitivity analysis.}
\label{arrival_times_sensitivity}
\resizebox{\textwidth}{!}{%
\begin{tabular}{lllllll}
\hline
$\Delta$ in average arrivals & Clinic time & Cost ratio & $\Delta_{\text{Idle Cost}}$ (\%) & $\Delta_{\text{Waiting Cost}}$ (\%) & $\Delta_{\text{Overtime Cost}}$ (\%) & $\Delta_{\text{Total Cost}}$ (\%) \\ \hline

5\%    & 150 & 1           & -0.01                    & 0.09                         & 0.00                  & 0.08               \\
        &                     & 5           & 0.01                     & 0.04                         & 0.01                  & 0.07               \\
        &                     & 10          & -0.02                    & 0.00                         & 0.00                  & -0.02              \\
        & 210 & 1           & -0.02                    & 0.07                         & 0.00                  & 0.05               \\
        &                      & 5           & 0.06                     & 0.05                         & 0.00                  & 0.11               \\
        &                      & 10          & 0.05                     & 0.02                         & 0.00                  & 0.07               \\
        & 270 & 1           & -0.01                    & 0.16                         & 0.00                  & 0.15               \\
        &                      & 5           & 0.01                     & 0.02                         & 0.00                  & 0.03               \\
        &                      & 10          & 0.02                     & 0.01                         & 0.00                  & 0.03               \\ \hline
     Average & & & 0.01                     & 0.05                         & 0.00                  & 0.06               \\ \hline
10\%     & 150 & 1           & 0.45                     & 0.32                         & 0.28                  & 1.06               \\
        &                      & 5           & 0.32                     & 0.05                         & 0.06                  & 0.43               \\
        &                      & 10          & 0.33                     & 0.05                         & 0.04                  & 0.42               \\
        & 210 & 1           & 0.55                     & 0.26                         & 0.00                  & 0.82               \\
        &                      & 5           & 0.40                     & 0.05                         & 0.00                  & 0.45               \\
        &                      & 10          & 0.45                     & 0.07                         & 0.00                  & 0.52               \\
        & 270 & 1           & 0.60                     & 0.28                         & 0.00                  & 0.88               \\
        &                      & 5           & 0.40                     & 0.03                         & 0.00                  & 0.44               \\
        &                      & 10          & 0.44                     & 0.03                         & 0.00                  & 0.47               \\ \hline
 Average & &  & 0.44                     & 0.13                         & 0.04                  & 0.61               \\ \hline
15\%     & 150 & 1           & 0.13                     & 0.34                         & 0.05                  & 0.52               \\
        &                      & 5           & 0.18                     & 0.10                         & 0.06                  & 0.34               \\
        &                      & 10          & 0.13                     & 0.05                         & 0.01                  & 0.19               \\
        & 210 & 1           & 0.19                     & 0.35                         & 0.00                  & 0.55               \\
        &                      & 5           & 0.14                     & 0.11                         & 0.00                  & 0.25               \\
        &                      & 10          & 0.08                     & 0.03                         & 0.00                  & 0.11               \\
        & 270 & 1           & 0.19                     & 0.31                         & 0.00                  & 0.49               \\
        &                      & 5           & 0.15                     & 0.09                         & 0.00                  & 0.24               \\
        &                      & 10          & 0.18                     & 0.04                         & 0.00                  & 0.23               \\ \hline
      Average & & & 0.15                    & 0.16                         & 0.01                  & 0.32  \\   \hline         
\end{tabular}
}

\end{table}

\begin{table}[htp]
\caption{No-show probability sensitivity analysis.}
\label{no_show_sensitivity}
\resizebox{\textwidth}{!}{%
\begin{tabular}{lllllll}
\hline
$\Delta$ in no-show probability & Clinic time & Cost ratio & $\Delta_{\text{Idle Cost}}$ (\%) & $\Delta_{\text{Waiting Cost}}$ (\%) & $\Delta_{\text{Overtime Cost}}$ (\%) & $\Delta_{\text{Total Cost}}$ (\%) \\ \hline
5\%      & 150  & 1 & -0.06 & 0.09 & 0.03 & 0.06 \\
         &      & 5 & -0.09 & 0.07 & 0.06 & 0.04 \\
         &      & 10 & -0.46 & -0.84 & -1.21 & -2.51 \\
         & 210  & 1 & -0.08 & 0.04 & 0.00 & -0.04 \\
         &      & 5 & -0.08 & 0.02 & 0.00 & -0.06 \\
         &      & 10 & 0.06 & -2.12 & -0.02 & -2.08 \\
         & 270  & 1 & -0.03 & -0.03 & 0.00 & -0.06 \\
         &      & 5 & -0.01 & -0.02 & 0.00 & -0.03 \\
         &      & 10 & -0.25 & -1.82 & 0.00 & -2.07 \\ \hline
Average  &  &  & -0.11 & -0.51 & -0.13 & -0.75 \\ \hline
10\%     & 150  & 1 & 0.15 & -0.06 & 0.03 & 0.12 \\
         &      & 5 & 0.17 & -0.13 & 0.01 & 0.05 \\
         &      & 10 & -0.34 & -0.96 & -1.23 & -2.53 \\
         & 210  & 1 & 0.01 & 0.05 & 0.00 & 0.06 \\
         &      & 5 & 0.11 & -0.10 & 0.00 & 0.01 \\
         &      & 10 & 0.29 & -2.23 & -0.02 & -1.96 \\
         & 270  & 1 & -0.06 & 0.07 & 0.00 & 0.00 \\
         &      & 5 & 0.01 & 0.02 & 0.00 & 0.02 \\
         &      & 10 & -0.08 & -1.77 & 0.00 & -1.84 \\ \hline
Average  &  &  & 0.03 & -0.57 & -0.13 & -0.67 \\ \hline
15\%     & 150  & 1 & 0.48 & -0.15 & 0.02 & 0.35 \\
         &      & 5 & 0.51 & -0.30 & 0.01 & 0.22 \\
         &      & 10 & 0.13 & -1.22 & -1.34 & -2.44 \\
         & 210  & 1 & 0.33 & -0.13 & 0.00 & 0.20 \\
         &      & 5 & 0.26 & -0.20 & 0.00 & 0.06 \\
         &      & 10 & 0.49 & -2.31 & -0.02 & -1.84 \\
         & 270  & 1 & 0.25 & -0.13 & 0.00 & 0.13 \\
         &      & 5 & 0.40 & -0.24 & 0.00 & 0.16 \\
         &      & 10 & -0.06 & -1.83 & 0.00 & -1.88 \\ \hline
Average  &  &  & 0.31 & -0.72 & -0.15 & -0.56 \\ \hline
\end{tabular}
}
\end{table}

\section{Conclusion}\label{AS:CC}
\noindent Our research introduces a novel approach to address the pressing challenges of healthcare appointment scheduling through an efficient stochastic model. \hll{Although our exposition and experiments focus on single-provider outpatient consultation sessions, such as primary care and specialty clinic visits, the proposed framework extends to any appointment-based service in which a single server sees scheduled customers under uncertain durations, punctuality, and attendance.} The model adeptly navigates the complexities in patient scheduling by considering the intrinsic uncertainties tied to service times, patient unpunctuality, and no-shows. In addition, these uncertainties are characterized by their patient-and-time-dependent distributions which are a critical aspect that previous models have often oversimplified due to the modeling and solution complexity. We showed that our model is efficient in solving large-scale instances optimally within a reasonable computational time. The proposed model captured an exponential number of scenarios while maintaining a polynomial number of variables and constraints for the first time in this field. Our novel integration of uncertain variables at an individualized time-dependent level has revealed significant operational efficiencies, demonstrated by a notable reduction in total clinic costs. Specifically, we observed a comprehensive cost reduction of 34\% attributed to the consideration of patient-dependent service times. Furthermore, the model's sensitivity to patient-and-time-dependent unpunctuality and no-show probabilities led to additional cost reductions of 12\% and 67\%, respectively. These findings highlight the critical importance of embracing the complexity of patient behavior patterns in scheduling models to optimize healthcare delivery and operational efficacy. Our findings highlight the vital importance of an optimized resource allocation, which maintains system efficiency without overburdening healthcare providers or compromising patient satisfaction.

We introduced personalized reminders as a strategic approach to reduce no-shows within our model framework that showed the profound impact of personalized communication in mitigating no-show rates. This consideration resulted in a further 23\% reduction in total costs. The obtained strategic insight highlights the transformative potential of using data-driven, patient-centric communication strategies to enhance clinic efficiency and patient engagement.

Our model not only challenges the conventional paradigms of healthcare scheduling but also sets a new benchmark for integrating patient-specific data to drive operational improvements. The detailed understanding and application of stochastic variables, reflective of real-world patient behaviors, enhance the model's practical relevance and applicability across diverse healthcare settings. By considering a wider array of factors, our model proved how a patient-and-time-dependent appointment scheduling approach can significantly enhance healthcare service efficiency and result in considerable cost savings and improved patient satisfaction.

\vspace{10pt}

\section*{\small Data availability}

\vspace{-5pt}

\noindent Data used in this research are available upon request.

\section*{\small Conflict of Interest}

\vspace{-5pt}

\noindent The authors declare that there is no conflict of interest.



\bibliographystyle{pomsref}

 \let\oldbibliography\thebibliography
 \renewcommand{\thebibliography}[1]{%
    \oldbibliography{#1}%
    \baselineskip14pt 
    \setlength{\itemsep}{10pt}
 }
\bibliography{ref1}




%
%
%

\newpage

\date{ }

\vspace{0 pt}

\numberwithin{equation}{section}
\renewcommand\theequation{S.\arabic{equation}}

\maketitle

\vspace*{20pt}

\begin{center}
    \Large
    Supplementary material for "Stochastic appointment scheduling with patient-and-time-dependent probability distributions"
\end{center} 

\vspace{20 pt}

\setcounter{page}{1}
\setcounter{section}{0}
\section{Proof of Theorem 1}

\noindent To establish the validity of Model (M1), we need to confirm the accuracy of the following statements:

\begin{enumerate}
\item  Variables $w_{it_i}$ and $z_{if_it_{i+1}}$ are well-defined by the proposed model.

\item  Objective function (1) computes the expected total cost correctly.
\end{enumerate}

\noindent \textbf{Proof of Statement 1:} Constraints (2) and (6) ensure the well-definition of the variables $w_{it_i}$. The proof of consistency between the values of variables $z_{if_it_{i+1}}$ and their formal definitions is established through an induction method. We first study the case $i\mathrm{=1}$ as the basis of induction. In this case, constraint (3) results in $z_{000}\mathrm{=}w_{\mathrm{1}0}=1$ that coincides with the definition of $z_{000}$ as we expect the appointment of dummy patient 0 to finish at 0 and the next patient $i=1$ to be scheduled at the beginning of the scheduling horizon. We note that we have $w_{\mathrm{1}0}=1$ based on constraint (2). 

\noindent As the inductive step, we suppose that for a fixed $i\mathrm{-}\mathrm{1}$, the model has computed all variables $z_{(i-1)f_{i-1}t_i}\ \forall f_{i-1}\in F_{i-1},t_i\in T_i$ properly. In the following, supposing that ${\tilde{y}}_i$ is a random variable representing the finish time of appointment for patient $i,\ $we first show that all variables $z_{if_it_{i+1}}\ \forall f_i\in F_i\backslash \{f^*\},t_{i+1}\in T_{i+1}$ are well-defined and then we discuss how it can be justified for $f_i=f^*,t_{i+1}\in T_{i+1}=\{.\}$

\begin{align}
z_{if_it_{i+1}} &= \mathrm{Pr}\left(\left(w_{(i+1)t_{(i+1)}}=1\right) \wedge \left({\tilde{y}}_i=f_i\right)\right) \nonumber \\
&= \mathrm{Pr}\left(w_{(i+1)t_{(i+1)}}=1\right)\mathrm{Pr}\left({\tilde{y}}_i=f_i\right) \nonumber \\
&= w_{(i+1)t_{(i+1)}}\mathrm{Pr}\left({\tilde{y}}_i=f_i\right) \nonumber \\
&= w_{(i+1)t_{(i+1)}}\sum_{t_i\in T_i} \mathrm{Pr}\left(w_{it_i}=1\right)\mathrm{Pr}\left({\tilde{y}}_i=f_i\mid w_{it_i}=1\right). \label{Aeq.1}
\end{align}

In the recent equations, we used the assumption that the two events $\left(w_{\mathrm{(}i\mathrm{+1)}t_{(i+1)}}\mathrm{=1}\right)$ and $\left({\tilde{y}}_i\mathrm{=}f_i\right)$ are independent. This assumption holds because $w_{\mathrm{(}i\mathrm{+1)}t_{(i+1)}}$ is not subject to uncertainty and is determined by the decision maker's choice. According to the latest formula, if $w_{\mathrm{(}i\mathrm{+1)}t_{(i+1)}}\mathrm{=0}$ then it follows that $z_{if_it_{i+1}}\mathrm{=0}$ which aligns with constraint (3). Consequently, we can exclude $w_{\mathrm{(}i\mathrm{+1)}t_{(i+1)}}$ from the final expression and proceed as below:
\begin{align}
z_{if_it_{i+1}} &= \sum_{t_i\in T_i} \mathrm{Pr}\left(w_{it_i}=1\right)\mathrm{Pr}\left(\tilde{y}_i=f_i \mid w_{it_i}=1\right)  \label{Aeq1.1}\\
&= \sum_{t_i\in T_i} w_{it_i} \left[\sum_{f_{i-1}\in F_{i-1}} \mathrm{Pr}\left(\tilde{y}_{(i-1)}=f_{(i-1)}\right)\mathrm{Pr}\left(\tilde{y}_i=f_i \mid \left(w_{it_i}=1\right) \wedge \left(\tilde{y}_{(i-1)}=f_{(i-1)}\right)\right)\right] \label{Aeq1.2}\\
&= \sum_{t_i\in T_i} \left[\sum_{f_{i-1}\in F_{i-1}} w_{it_i}\mathrm{Pr}\left(\tilde{y}_{(i-1)}=f_{(i-1)}\right)\mathrm{Pr}\left(\tilde{y}_i=f_i \mid \left(w_{it_i}=1\right) \wedge \left(\tilde{y}_{(i-1)}=f_{(i-1)}\right)\right)\right] \label{Aeq1.3}\\
&= \sum_{t_i\in T_i} \left[\sum_{f_{i-1}\in F_{i-1}} w_{it_i}\mathrm{Pr}\left(\tilde{y}_{(i-1)}=f_{(i-1)}\right)\mathrm{Pr}\left(\max(t_i,f_{i-1})+\tilde{d}_i=f_i\right)\right] \label{Aeq1.4}\\
&= \sum_{t_i\in T_i} \left[\sum_{f_{i-1}\in F_{i-1}} \mathrm{Pr}\left(w_{it_i}=1\right)\mathrm{Pr}\left(\tilde{y}_{(i-1)}=f_{(i-1)}\right)\mathrm{Pr}\left(\max(t_i,f_{i-1})+\tilde{d}_i=f_i\right)\right] \label{Aeq1.5}\\
&= \sum_{t_i\in T_i} \left[\sum_{f_{i-1}\in F_{i-1}} \mathrm{Pr}\left(\left(w_{it_i}=1\right) \wedge \left(\tilde{y}_{(i-1)}=f_{(i-1)}\right)\right)\mathrm{Pr}\left(\max(t_i,f_{i-1})+\tilde{d}_i=f_i\right)\right] \label{Aeq1.6}\\
&= \sum_{t_i\in T_i} \left[\sum_{f_{i-1}\in F_{i-1}} z_{(i-1)f_{i-1}t_i} \mathrm{Pr}\left(\max(t_i,f_{i-1})+\tilde{d}_i=f_i\right)\right] \label{Aeq1.7}\\
&= \sum_{t_i\in T_i} \sum_{f_{i-1}\in F_{i-1}} z_{(i-1)f_{i-1}t_i} \mathrm{Pr}\left(\max(t_i,f_{i-1})+\tilde{d}_i=f_i\right) \label{Aeq1.8}
\end{align}

Also, relation (\ref{Aeq1.5}) is valid as $w_{it_i}\mathrm{=Pr}\mathrm{}\left(w_{it_i}\mathrm{=1}\right)$ holds, given that $w_{it_i}$ is a binary variable determined by the decision maker and lacks any uncertain nature. Relation (\ref{Aeq1.6}) is valid because the events $w_{it_i}\mathrm{=1}$ and $\mathrm{Pr}\mathrm{}\left({\tilde{y}}_{\left(i-1\right)}\mathrm{=}f_{\left(i-1\right)}\right)$ are independent. Finally, the validity of relation (\ref{Aeq1.7}) aligns with the definition of $z_{(i-1)f_{i-1}t_i}$ (the basis of induction), ensuring that constraints (3)-(4) correctly define variables $z_{if_it_{i+1}}\ $ for$\ f_i\in F_i\backslash \{f^*\},t_{i+1}\in T_{i+1}$. 
\\

\noindent \textbf{Proof of Statement 2: }To verify that the objective function is accurately computed, we focus on the representation of random cost variables for patient  $i$ including ${\widetilde{\mathrm{c}}}^{waiting}_i$, ${\widetilde{\mathrm{c}}}^{idle}_i$, and ${\widetilde{\mathrm{c}}}^{overtime}_i$. These variables are contingent upon the scheduled appointment times and the uncertain parameters of service time.
\\
\begin{align}
\noindent \text{Total expected cost} &= \sum_{i \in I_0} E\left(\tilde{c}^{waiting}_i + \tilde{c}^{idle}_i + \tilde{c}^{overtime}_i \right) \nonumber \\
&= \sum_{i \in I_0} \left(E\left[\tilde{c}^{waiting}_i\right] + E\left[\tilde{c}^{idle}_i\right] + E\left[\tilde{c}^{overtime}_i\right]\right) \nonumber \\
&= \sum_{i \in I_0} \sum_{f_i \in F_i} \sum_{t_{i+1} \in T_{i+1}} \Pr\left(\left(\tilde{y}_i = f_i\right) \wedge \left(w_{(i+1)(t_{i+1})} = 1\right)\right) \nonumber \\
&\quad \times \left(E\left[\tilde{c}^{waiting}_i \mid \left(\tilde{y}_i = f_i\right) \wedge \left(w_{(i+1)(t_{i+1})} = 1\right)\right] + \right. \nonumber \\
&\quad \left. E\left[\tilde{c}^{idle}_i \mid \left(\tilde{y}_i = f_i\right) \wedge \left(w_{(i+1)(t_{i+1})} = 1\right)\right] + \right. \nonumber \\
&\quad \left. E\left[\tilde{c}^{overtime}_i \mid \left(\tilde{y}_i = f_i\right) \wedge \left(w_{(i+1)(t_{i+1})} = 1\right)\right]  \right) \nonumber \\
&= \sum_{i \in I_0} \sum_{f_i \in F_i} \sum_{t_{i+1} \in T_{i+1}} z_{if_it_{i+1}} \nonumber \\
&\quad \times \left(c^p_{if_it_{i+1}} + c^d_{if_it_{i+1}} + c^o_{if_it_{i+1}} \right) 
\end{align}

\noindent Therefore, Statement 2 holds too and Theorem 1 is valid.

\section{Appointment Scheduling Example} \label{chap:Example}

\noindent He we illustrate our model through a numerical example. This example involves two patients both of whom have already been assigned their appointment slots. Thus, Constraint (2) regarding single appointment assignments per patient is satisfied and we have $w_{\mathrm{10}}\mathrm{=1,\ }w_{\mathrm{2(30)}}\mathrm{=1}$. The service times for Patient 1 are (20, 40) minutes with corresponding probabilities of (0.5, 0.5), and for Patient 2 we have service times of (20, 30) minutes and corresponding probabilities of (0.5, 0.5). The model captures the probability of different finish times for each patient considering their respective service time distributions. This information is presented in Table \ref{exampletable} which outlines the probability of each patient finishing their appointment at various times.

\begin{table}[htbp]
\centering
\caption{Probability of finish times for each patient}
\label{exampletable}
\begin{tabular}{|c|c|c|c|} 
\hline 
Line Number & Patient & Finish time & Probability \\ 
\hline 
1 & 0 & 0 & 1 \\ 
\hline 
2 & 1 & 20 & 0.5 \\ 
\cline{1-1} \cline{3-4}
3 &  & 40 & 0.5 \\ 
\hline 
4 & 2 & 50 & 0.25 \\ 
\cline{1-1} \cline{3-4}
5 &  & 60 & 0.5 \\ 
\cline{1-1} \cline{3-4}
6 &  & 70 & 0.25 \\ 
\hline 
\end{tabular}
\end{table}

\noindent First, we schedule Patient 1 in the first appointment slot. Following Constraint (3) we derive the following set of constraints ($\theta =10$):

\begin{align}
z_{000} &= 1 \label{B.1} \\
z_{00(t_1)} &= 0 \quad \text{for } t_1 \in T_1\backslash \{0\} \label{B.2}
\end{align}

\noindent which validates the first line of Table 1.

\noindent Considering Constraint (4), we also derive the following constraint:
\begin{align}
z_{000} \Pr(\max(0,0) + \tilde{d}_1 = 20) &= z_{1(20)(20)} + z_{1(20)(30)} + z_{1(20)(40)} + z_{1(20)(50)} \nonumber \\
&\quad + z_{1(20)(60)} + z_{1(20)(70)} + z_{1(20)(80)} + z_{1(20)(90)}.
\label{B.3}
\end{align}

\noindent Upon simplification, we obtain:
\begin{align}
    0.5 &= z_{1(20)(20)} + z_{1(20)(30)} + z_{1(20)(40)} + z_{1(20)(50)} \nonumber \\
    &\quad + z_{1(20)(60)} + z_{1(20)(70)} + z_{1(20)(80)} + z_{1(20)(90)} \label{B.4}
\end{align}

\noindent From Constraint (3) for Patient 2, we have:
\begin{align}
    \sum_{f_1 \in F_1} z_{1f_1(t_1)} &= 0 \quad \text{for } t_1 \in T_1 \setminus \{30\} \label{B.5}
\end{align}

\noindent Thus, from (\ref{B.5}) and (\ref{B.4}) we have:
\begin{align}
    0.5 &= z_{1(20)(30)} \label{B.6}
\end{align}
which validates the second line of Table 1.

\noindent Considering Constraint (4) for Patient 1 and $f = 40$, we have:
\begin{align}
z_{000} \Pr(\max(0,0) + \tilde{d}_1 = 40) &= z_{1(40)(20)} + z_{1(40)(30)} + z_{1(40)(40)} + z_{1(40)(50)} \nonumber \\
&\quad + z_{1(40)(60)} + z_{1(40)(70)} + z_{1(40)(80)} + z_{1(40)(90)}.
\label{B.7}
\end{align}

\noindent Upon simplification, we obtain:
\begin{align}
    0.5 &= z_{1(40)(20)} + z_{1(40)(30)} + z_{1(40)(40)} + z_{1(40)(50)} \nonumber \\
    &\quad + z_{1(40)(60)} + z_{1(40)(70)} + z_{1(40)(80)} + z_{1(40)(90)}. \label{B.8}
\end{align}
Also, by considering Constraint (3) for Patient 2, we have $\sum_{f_1 \in F_1} z_{1f_1(t_1)} = 0$ for $t_1 \in T_1 \setminus \{30\}$. Now, based on (\ref{B.5}) and (\ref{B.8}), we can show that:
\begin{align}
    0.5 &= z_{1(40)(30)} \label{B.9}
\end{align}
This validates the third line of Table \ref{exampletable}. Furthermore, considering Constraint (4) for Patient 2 and $f=50$, we derive:
\begin{align}
    & z_{1(20)(30)} \Pr(\max(30,20) + \tilde{d}_2 = 50) + \sum_{t_i \in T_i \setminus \{30\}} z_{1(20)(t_i)} \Pr(\max(t_i,20) + \tilde{d}_2 = 50) \nonumber \\
    & + \sum_{t_i \in T_i} z_{1(40)(t_i)} \Pr(\max(t_i,40) + \tilde{d}_2 = 50) = z_{2(50)(\cdot)} \label{B.10}
\end{align}

\noindent In Constraint (\ref{B.10}), the expression $\sum_{t_i \in T_i \setminus {30}} z_{1(20)(t_i)} \Pr(\max(t_i,20) + \tilde{d}_2 = 50)$ is equal to zero considering (\ref{B.5}). The expression $\sum_{t_i \in T_i} z_{1(40)(t_i)} \Pr(\max(t_i,40) + \tilde{d}_2 = 50)$ is equal to zero since $\max(t_i,40) + \tilde{d}_2 = 50$ is not possible. Therefore, we have:
\begin{align}
    0.5 \cdot 0.5 &= z_{2(50)(.)} \label{B.11}
\end{align}

\noindent This validates the fourth line of Table \ref{exampletable}. Further, considering Constraint (\ref{B.4}) for Patient 2 and $f=60$ we have:

\begin{align}
    & z_{1(20)(30)}\Pr(\max(30,20)+\tilde{d}_2=60) + \sum_{t_i\in T_{i}\backslash \{30\}}{z_{1(20)(t_i)}\Pr(\max(t_i,20)+\tilde{d}_2=60)} \nonumber\\
    &+ z_{1(40)(30)}\Pr(\max(30,40)+\tilde{d}_2=60) \nonumber\\
    &+ \sum_{t_i\in T_{i}\backslash \{30\}}{z_{1(40)(t_i)}\Pr(\max(t_i,40)+\tilde{d}_2=60)} = z_{2(60)(.)} \label{B.12}
\end{align}

Given (\ref{B.12}), expressions  $\sum_{t_i\in T_{i}\backslash \{30\}}{z_{1(20)(t_i)}\Pr(\max(t_i,20)+\tilde{d}_2=60)}$ and \\ $\sum_{t_i\in T_i\backslash \{30\}}{z_{1(40)(t_i)}\Pr(\max(t_i,40)+\tilde{d}_2=60)}$ are equal to zero. Thus, we have:
\begin{align}
    0.5 \cdot 0.5 + 0.5 \cdot 0.5 &= z_{2(60)(.)} \label{B.13}
\end{align}

\noindent which validates the fifth line of Table \ref{exampletable}. Considering Constraint (4) for Patient 2 and $f=70$ we have:
\begin{align}
    & \sum_{t_i\in T_{i}\backslash \{30\}}{z_{1(20)(t_i)}\Pr(\max(t_i,20)+\tilde{d}_2=70)} + z_{1(20)(30)}\Pr(\max(30,20)+\tilde{d}_2=70) \nonumber\\
    & + z_{1(40)(30)}\Pr(\max(30,40)+\tilde{d}_2=70) + \sum_{t_i\in T_{i}\backslash \{30\}}{z_{1(40)(t_i)}\Pr(\max(t_i,40)+\tilde{d}_2=70)} \nonumber\\
    & + \sum_{t_i\in T_i}{z_{1(f\neq 40)(t_i)}\Pr(\max(t_i,f\neq 40)+\tilde{d}_2=70)} = z_{2(50)(.)} \label{B.14}
\end{align}

\noindent 
In above equality, $\sum_{t_i \in T_{i\setminus {30}}} z_{1(20)(t_i)} \Pr(\max(t_i,20) + \tilde{d}_2 = 70)$ is equal to zero considering (\ref{B.5}). 

In addition, $z_{1(20)(30)} \Pr(\max(30,20) + \tilde{d}2 = 70)$ is equal to zero since $\max(30,20) + \tilde{d}2 = 70$ is not possible. Moreover, $z_{1(40)(30)} \Pr(\max(30,40) + \tilde{d}2 = 70)$ is equal to zero considering (\ref{B.5}). Finally, $\sum_{t_i \in T{i\setminus {30}}} z_{1(40)(t_i)} \Pr(\max(t_i,40) + \tilde{d}_2 = 70)$ is equal to zero since constraint (3) for Patient 2 shows $z_{1(20)(30)}+z_{1(40)(30)}+\sum_{t_i \in T_i} z_{1(f\neq20)(30)} =1$, which together with $0.5=z_{1(20)(30)}$ and $0.5=z_{1(20)(40)}$ proves $\sum_{t_i \in T_i} z_{1(f\neq20)(30)} =0$. Thus we have:

\begin{align}
    0.5 &= z_{2(50)(.)} \label{B.15}
\end{align}

which validates the last line of Table \ref{exampletable}.

\end{document}